\documentclass[11pt]{article}%
\usepackage{amsfonts}
\usepackage{mathrsfs}
\usepackage{amsmath}
\usepackage{amssymb}
\usepackage{graphicx}
\usepackage{latexsym}
\usepackage{indentfirst}%
\usepackage{color,todonotes,ulem} %yy
\hyphenpenalty=1500  \numberwithin{equation}{section}
\newtheorem{thm}{Theorem}[section]
\newtheorem{lema}[thm]{Lemma}
\newtheorem{prop}[thm]{Proposition}
\newtheorem{rmk}[thm]{Remark}

\title%[Shock profiles for full Navier--Stokes--Poisson equations]
{Existence and stability of shock profiles for a one-dimensional full compressible Navier--Stokes--Poisson equations}
\author{\sc Yeping Li$^{1}$ and Yuan Yuan$^{2}$\thanks{Corresponding author. {\it E-mail address:
yyuan2102@m.scnu.edu.cn}}\\
$^1${\footnotesize \it School of Mathematics and Statistics, Nantong
University,
Nantong 226019, P.R. China.}\\
$^2${\footnotesize \it School of Mathematical Sciences, South China
Normal University, Guangzhou 510631,  P. R. China.}}

\date{}

\begin{document}

\maketitle

\noindent{\bf Abstract:}  This paper studies the existence and asymptotic stability of traveling shock waves for the one-dimensional full compressible Navier--Stokes--Poisson equations, which model viscous heat-conducting ions in a self-consistent electrostatic field. We first prove the existence and uniqueness of a smooth shock profile of small amplitude by the center manifold theorem, and then establish its large-time asymptotic stability under suitably small smooth perturbations. The proof is based on an energy estimate for a specially constructed integrated system derived from the full compressible Navier--Stokes--Poisson equations. To carry out this estimate, we introduce suitable anti-derivative variables that incorporate the effects of the self-consistent electrostatic potential and heat conduction, thereby reformulating the original system as a dissipative integrated one. The subsequent energy estimates are completed by carefully handling the induced couplings.

\


\


\section{Introduction}
This article is concerned with the one-dimensional full compressible
Navier-Stokes-Poisson (denoted as NSP in the sequel) system for
viscous heat-conductivity ions motion, which reads as
\begin{eqnarray}
\left\{
\begin{array}{lcr}
\rho_t+(\rho
u)_x=0,\\
(\rho u)_t+(\rho u^2+P)_x=\mu
u_{xx}-\rho\phi_x,\\
(\rho(e+\frac{1}{2}u^2))_t+(\rho
u(e+\frac{1}{2}u^2)+Pu)_x=(\kappa\theta_x+\mu uu_x)_x-\rho
u\phi_x,\\
-\lambda^2\phi_{xx}=\rho-e^{\phi}.
\end{array}\label{model11}\right.
\end{eqnarray}
where $t\geq0$ is the time variable and $x\in\mathbb{R}$ is spatial
variable. The functions $\rho(t,x)>0, u(t,x),\theta(t,x)>0$ and
$\phi=\phi(t,x)$ stand for the density, the velocity, the absolute
temperature and the electrostatic potential, respectively. The
positive constants $\nu,\kappa$ and $\lambda$ are the viscosity
coefficients of the fluid, the heat conductivity coefficient and the
Debye length, respectively.  $P$ is the pressure, and $e$ is the
internal energy.

\vskip .1in The compressible NSP system has several physical applications. In the ion model considered here, the electron density is determined by the electrostatic potential through the relation $\rho_e=e^{\phi}$. Related NSP models are used for charged-particle transport in semiconductor devices and plasmas, and also arise in models of self-gravitating viscous gases. If the ions are fixed, the corresponding NSP system describes the motion of electrons. If heat conduction is neglected, one obtains the isentropic compressible NSP system, while a two-fluid NSP model is used when both ions and electrons are treated dynamically. We refer to \cite{C2,D,J,SM} for further physical background. In this paper, we consider an ideal polytropic gas with
\begin{equation}
	P=R\rho \theta,\quad e=\frac{R}{\gamma-1}\theta,\label{relation11}
\end{equation}
where $R>0$ is the gas constant and $\gamma>1$ is the adiabatic exponent.

\vskip .1in 
The existence theory for compressible NSP equations has been extensively studied in many settings. We refer to \cite{CHS,DWYZ,HL,LMZ,LLZ,LXZ,SX1,TYZZ} for strong or classical solutions and to \cite{D2,DF,KS} for weak solutions. Since the present paper concerns large-time behavior, we next recall results that are more directly related to asymptotic dynamics. For the NSP system describing electrons (with fixed ions) near a non-vacuum constant state, Li, Matsumura and Zhang \cite{LMZ} obtained global existence and algebraic decay for the three-dimensional isentropic system. Optimal decay estimates were later derived by Li and Zhang \cite{LZ1}, while Wang \cite{W} obtained the optimal decay by energy estimates. Pointwise estimates based on the Green function in $\mathbb{R}^n(n \geq3)$ were established by Wang and Wu \cite{WW}, and further decay results in higher dimensions were obtained in \cite{WW1}. Related results for the full NSP system of electrons and the bipolar NSP system can be found in \cite{LTZ,TX,WW2,ZLZ}.

\vskip .1in There are also substantial developments on the stability and large-time behavior around nonlinear wave patterns for the compressible NSP system, such as the stationary wave, the viscous shock wave, the rarefaction wave and the contact discontinuity. 
Duan and Liu \cite{DL} obtained the stability of rarefaction waves for the one-dimensional one-fluid isentropic NSP system, and Duan, Liu and Zhang \cite{DLZ} further established the stability of treated viscous shock waves for the same ion model. 
The stability of rarefaction waves with large initial perturbations were proved by Zhang, Zhao and Zhao \cite{ZZZ,Z1} recently. Shock profiles and rarefaction waves under space-periodic perturbations were studied in \cite{LMY}, while Kang, Kwon and Shim \cite{KKS} established stability of isothermal shock profiles by the $a$-contraction method; see also \cite{S} for a shock--rarefaction composite wave. 
Li, Luo and Wu \cite{LLW} and Wu \cite{Wuxiaochun} considered the time-asymptotically nonlinear stability of the planar rarefaction wave and the planar shock wave of the initial value problem for the three-dimensional isentropic compressible NSP system for the ions, respectively.
For the out-flow problem of the compressible NSP system, the existence and the stability of stationary solutions to the system of electrons were shown by Hong and Kim \cite{HK}, and Wang, Zhang and Zhang \cite{WZZ}, and Li and Zhu \cite{LZ} considered a superposition of a stationary solution and a rarefaction wave to the system of ions.

\vskip .1in However, for the full compressible NSP equations for ions, results on the large-time behavior toward hyperbolic waves are much more limited. To the best of our knowledge, apart from the independent work of Duan and Li \cite{Duan.L2026} discussed below, there are four closely related results, and \textit{none concerns shock profiles}. Hong \cite{H} studied the existence and stability of stationary solutions for an inflow problem of the full compressible NSP system for electrons. Liu, Yin and Zhu \cite{LYZ} treated a contact discontinuity, and Ruan, Yin and Zhu \cite{RYZ} studied a superposition of a rarefaction wave and a contact discontinuity for the ion system with free boundary. Li, Chen and Chen \cite{LCC} considered a planar rarefaction wave for the three-dimensional full compressible NSP system. 
In an independent work appearing at essentially the same time, Duan and Li \cite{Duan.L2026} also established the existence of small-amplitude shock profiles by the center manifold method. Their profile construction, like ours, exploits the conservative structure induced by the Poisson equation to reduce the traveling-wave equations to a finite-dimensional first-order system. The approaches to nonlinear stability, however, are substantially different: they employ the $a$-contraction method with shifts and relative entropy without imposing a zero-mass condition, whereas our analysis is based on a specially constructed anti-derivative system and energy estimates under the zero-integral condition.
Related results for two-fluid compressible NSP systems can be found in \cite{CGYZ,DLYZ,DY,HSW,YZZ}.

\vskip .1in The main purpose of this paper is to study the existence
and the asymptotic stability of shock profiles of the full
compressible NSP system (\ref{model11}).%for viscous and
%heat-conductivity ions motion. That is, we consider the effect of
%the variable temperature for the full compressible NSP system with
%the additional evolution equation of the temperature function, and
%we are going to study the asymptotic stability of traveling waves
%with shock profile of the one-dimensional full NSP equations .
Now, we prescribe the initial values to the model (\ref{model11}) as
follows:
\begin{eqnarray}
(\rho,u,\theta)(0,x)=(\rho_0,u_0,\theta_0)(x),\label{initial11}
\end{eqnarray}
with
\begin{eqnarray}
\lim_{x\rightarrow\pm\infty}(\rho_0,u_0,\theta_0)(x)=(\rho_\pm,
u_\pm, \theta_\pm),~~\rho_\pm>0,\theta_\pm>0.\label{initial12}
\end{eqnarray}
The far-field data of $\phi$ are given by
\begin{eqnarray}
\lim\limits_{x\to\pm\infty}\phi(t,x)=\phi_\pm\label{initial13}
\end{eqnarray}
with the \textit{quasi-neutral condition} %$\phi_\pm=\ln\rho_\pm$, namely, $\rho_\pm=e^{\phi_\pm}$ at $x=\pm\infty$.
\begin{equation}
    \label{quasi-neutral}
    \phi_\pm=\ln\rho_\pm.
\end{equation}

For convenience, by setting $v=\frac1\rho,v_\pm=\frac{1}{\rho_\pm}$ and
$p=R\frac{\theta}{v}$, we can formulate the one-dimensional full
compressible NSP system (\ref{model11}) in the Lagrangian
coordinates:
\begin{eqnarray}
\left\{
\begin{array}{lcr}
v_t-u_x=0, \\
u_t+p_x=\mu(\frac{u_x}{v})_x-\frac{\phi_x}{v}, \\
(e+\frac{1}{2}u^2)_t+(pu)_x=(\frac{\kappa}{v}\theta_x+\frac{\mu}{v}uu_x)_x-u\frac{\phi_x}{v}, \\
-\lambda^2(\frac{\phi_x}{v})_x=1-ve^\phi,
\end{array}\right.\label{model12}
\end{eqnarray}
with the corresponding initial data and the far-field states
\begin{eqnarray}
\left\{
\begin{array}{lcr}
(v,u,\theta)(0,x)=(v_0,u_0,\theta_0)(x)\rightarrow(v_\pm,u_\pm,\theta_\pm)~{\rm
as}~x\rightarrow\pm\infty,\\
\lim\limits_{x\to\pm\infty}\phi(t,x)=\phi_\pm.\end{array}\right.\label{initial14}
\end{eqnarray}

The existence and dynamical stability of viscous shock waves are basic problems in conservation laws and viscous fluid equations. Gilbarg \cite{G1} constructed shock profiles for the one-dimensional compressible Navier--Stokes equations and studied their structure for small viscosity and heat conductivity. Conley and Smoller \cite{CS} used a topological method to construct weak shock profiles for viscous conservation laws. Majda and Pego \cite{MP} later connected shock admissibility with the structure of the viscosity matrix and constructed weak shock profiles for a broad class of viscous conservation laws by the center manifold theorem; see also Dafermos \cite{D0} for a detailed account of the shock structure problem. For dynamical stability, Matsumura and Nishihara \cite{MN} treated the isentropic Navier--Stokes equations, and Kawashima and Matsumura \cite{KM} extended the analysis to the heat-conductive case. Further results for general viscous conservation laws can be found in \cite{G2,LZ2,LZ3,MZ,P,SX}.

Motivated by \cite{DLZ,KM,MP,MN}, we are going to study the
existence and the asymptotic stability of traveling shock profiles of the full
compressible NSP system (\ref{model12}) with (\ref{initial14}).
General speaking, the large time behavior of solutions to the
one-dimensional Cauchy problem (\ref{model12}) and (\ref{initial14})
is determined by the far-field data $(v_\pm, u_\pm,
\theta_\pm,\phi_\pm)$. In the paper, we are concerned with only the
shock profile under the quasi-neutral condition
$\phi_\pm=\ln\frac1{v_\pm}$. We first recall the definition of shock
profiles briefly. Let $(v,u,\theta,\phi)$ be a smooth travelling
wave solution of (\ref{model12}) which depends only on the variable
$y=x-st$ with the wave speed $s$ to be determined later and connects
the far-fields $(v_\pm, u_\pm, \theta_\pm,\phi_\pm)$ at
$x=\pm\infty$ respectively. The profile is thereby governed by the
following system of ODEs:
\begin{eqnarray}
\left\{
\begin{array}{lcr}
-s\frac{\mathrm{d}v}{\mathrm{d}y}-\frac{\mathrm{d}u}{\mathrm{d}y}=0,\\
-s\frac{\mathrm{d}u}{\mathrm{d}y}+\frac{\mathrm{d}p}{\mathrm{d}y}
=\mu\frac{\mathrm{d}}{\mathrm{d}y}(\frac{1}{v}\frac{\mathrm{d}u}{\mathrm{d}y})-\frac1v\frac{\mathrm{d}\phi}{\mathrm{d}y},\\
-s\frac{\mathrm{d}(e+\frac{1}{2}u^2)}{\mathrm{d}y}+\frac{\mathrm{d}(pu)}{\mathrm{d}y}=\frac{\mathrm{d}}{\mathrm{d}y}(\frac\kappa
v\frac{\mathrm{d}\theta}{\mathrm{d}y}+\frac{\mu u}
{v}\frac{\mathrm{d}u}{\mathrm{d}y})-\frac uv\frac{\mathrm{d}\phi}{\mathrm{d}y},\\
-\lambda^2\frac{\mathrm{d}}{\mathrm{d}y}(\frac{1}{v}\frac{\mathrm{d}\phi}{\mathrm{d}y})=1-ve^{\phi},
\end{array}\right.\label{model13}
\end{eqnarray}
with
\begin{eqnarray}
\lim\limits_{y\to\pm\infty}(v,u,\theta,\phi)(y)=(v_\pm,u_\pm,\theta_\pm,\phi_\pm).\label{initial15}
\end{eqnarray}
Note that by using the fourth and the first equation of
(\ref{model13}), we see
\begin{eqnarray*}
-\frac1v\frac{\mathrm{d}\phi}{\mathrm{d}y}
=\frac{\mathrm{d}}{\mathrm{d}y}\big(\frac{\lambda^2}{2}(\frac{1}{v}\frac{\mathrm{d}\phi}{\mathrm{d}y})^2-\frac1v
-\frac{\lambda^2}v\frac{\mathrm{d}}{\mathrm{d}y}(\frac{1}{v}\frac{\mathrm{d}\phi}{\mathrm{d}y})\big),
\end{eqnarray*}
and
\begin{eqnarray*}
-\frac
uv\frac{\mathrm{d}\phi}{\mathrm{d}y}=\frac{\mathrm{d}}{\mathrm{d}y}\big(\frac{\lambda^2}{2}(\frac{1}{v}\frac{\mathrm{d}\phi}{\mathrm{d}y})^2u
-ue^\phi-sve^\phi+s\phi
+\frac{\lambda^2s}{2v}(\frac{\mathrm{d}\phi}{\mathrm{d}y})^2\big).
\end{eqnarray*}
Then we can rewrite  system (\ref{model13}) as
\begin{eqnarray}
\left\{
\begin{array}{lcr}
-s\frac{\mathrm{d}v}{\mathrm{d}y}-\frac{\mathrm{d}u}{\mathrm{d}y}=0,\\
-s\frac{\mathrm{d}u}{\mathrm{d}y}+\frac{\mathrm{d}}{\mathrm{d}y}(p+\frac1v)
=\frac{\mathrm{d}}{\mathrm{d}y}\big(\frac{\mu}{v}\frac{\mathrm{d}u}{\mathrm{d}y}
+\frac{\lambda^2}{2}(\frac{1}{v}\frac{\mathrm{d}\phi}{\mathrm{d}y})^2
-\frac{\lambda^2}v\frac{\mathrm{d}}{\mathrm{d}y}(\frac{1}{v}\frac{\mathrm{d}\phi}{\mathrm{d}y})\big),\\
-s\frac{\mathrm{d}}{\mathrm{d}y}\big(e+\frac{1}{2}u^2-ve^\phi+\phi
+\frac{\lambda^2}{2v}(\frac{\mathrm{d}\phi}{\mathrm{d}y})^2\big)
+\frac{\mathrm{d}}{\mathrm{d}y}\big((p+\frac1v)u\big)\\
~~~~~~~~~~~~~~~~~~=\frac{\mathrm{d}}{\mathrm{d}y}\big(\frac\kappa
v\frac{\mathrm{d}\theta}{\mathrm{d}y}+\frac{\mu u}
{v}\frac{\mathrm{d}u}{\mathrm{d}y}+\frac{\lambda^2}{2}(\frac{1}{v}\frac{\mathrm{d}\phi}{\mathrm{d}y})^2u\big),\\
-\lambda^2\frac{\mathrm{d}}{\mathrm{d}y}(\frac{\frac{\mathrm{d}\phi}{\mathrm{d}y}}{v})=1-ve^{\phi},
\end{array}\right.\label{model14}
\end{eqnarray}
which consists of three conservation laws together with the Poisson
equation for the electric potential. By integrating the first three
equations in (\ref{model14}) over $-\infty<y<+\infty$, then using
$\phi_\pm=-\ln v_\pm$, and further assuming that all the first-order
derivative terms vanish as $y=\pm\infty$, we obtain the
Rankine-Hugoniot (R-H) condition as follows:
\begin{eqnarray}
\left\{
\begin{array}{lcr}
-s(v_+-v_-)-(u_+-u_-)=0, \\
-s(u_+-u_-)+\frac{R\theta_++1}{v_+}-\frac{R\theta_-+1}{v_-}=0, \\
-s(\frac{R}{\gamma-1}(\theta_+-\theta_-)+\frac{1}{2}(u^2_+-u^2_-)
+\ln\frac1{v_+}-\ln\frac1{v_-})+\frac{(R\theta_++1)u_+}{v_+}-\frac{(R\theta_-+1)u_-}{v_-}=0.
\end{array}\right.\label{relation12}
\end{eqnarray}
We point out that the R-H condition (\ref{relation12}) is the same
as the one for the following full quasi-neutral type Euler system:
\begin{eqnarray}
\left\{
\begin{array}{lcr}
v_t-u_x=0, \\
u_t+(p+\frac1v)_x=0, \\
(e+\frac12u^2+\ln\frac1v)_t+((p+\frac1v)u)_x=0,
\end{array}\right.\label{model15}
\end{eqnarray}
which can be formally obtained by letting $\mu=\kappa=\lambda=0$ in
(\ref{model12}). Moreover, we assume that Lax's condition
\begin{eqnarray}
{\rm either}~\lambda_1^+<s<\lambda_1^-~{\rm
or}~\lambda_3^+<s<\lambda_3^-, \label{relation13}
\end{eqnarray}
where 
\begin{equation*}
\lambda_1^\pm=-\frac{\sqrt{R\gamma\theta^\pm+1}}{v^\pm},
~\lambda_2^\pm=0,~\lambda_3^\pm=\frac{\sqrt{R\gamma\theta^\pm+1}}{v^\pm} 
\end{equation*}
is three eigenvalues of (\ref{model15}). Note that
(\ref{relation13}) is equivalent to the shock profile with
compressibility, namely,
\begin{eqnarray*}
u_+<u_-,~{\rm or}~s(v_+-v_-)>0.
\end{eqnarray*}

In this paper, we focus on the $3$-shock profile with
$0<\lambda^+_3<s<\lambda^-_3$ and $v_+>v_-$. Now we state the result
about the existence of shock profile of (\ref{model12}) with fixed
constants $\mu>0, \kappa>0$ and $\lambda>0$ by the following
Theorem.
\begin{thm}
For given data $(v_-,u_-,\theta_-)$ with $v_-,\theta_->0$, there
exists a  positive constant $\delta_0$ such that if
$(v_+,u_+,\theta_+)$ satisfies (\ref{relation12}) and
$0<v_+-v_-<\delta_0$, then the problem (\ref{model13}) or
equivalently (\ref{model14}) has a unique (up to a shift) smooth
traveling wave solution of the form $(\bar{v}, \bar{u},
\bar{\theta}, \bar{\phi})(x-st)$ connecting the far fields $(v_\pm,
u_\pm, \phi_\pm, \theta_\pm)$ with \eqref{quasi-neutral}.
Moreover, there exist positive constants $ c, C, a$ and
$C_k(k=0,1,2,\cdots)$, such that for each $t>0$ and
$x\in\mathbb{R}$, it holds that
\begin{eqnarray}
s\bar{v}_x=-\bar{u}_x>0,~
-C\bar{v}_x\le\bar{\phi}_x,\bar\theta_x\le-c\bar{v}_x<0,\label{inequality11}
\end{eqnarray} and
\begin{eqnarray}
\big|\frac{\mathrm{d}^k}{\mathrm{d}x^k}(\bar{v}-v_\pm,
\bar{u}-u_\pm, \bar{\theta}-\theta_\pm, \bar{\phi}-\phi_\pm)\big|
\le C_k|v_+-v_-|^{k+1}e^{-a|v_+-v_-||x-st|}\label{inequality12}
\end{eqnarray}
for $x-st\lessgtr0$ and $k=0,1,2,\cdots$.
\end{thm}

Next, let us turn to the asymptotic stability of the smooth
traveling shock profile $(\bar{v}, \bar{u}, \bar{\theta},
\bar{\phi})$ obtained in Theorem 1.1. First, assume
\begin{eqnarray}
(v_0-\bar v,u_0-\bar u,\theta_0-\bar\theta)(x)\in
H^1(\mathbb{R}),\label{assumption11}
\end{eqnarray}
and the integrals
\begin{eqnarray}
&&\Phi_0(x)=\int_{-\infty}^{x}(v_0-\bar
v)(y)\mathrm{d}y,~\Psi_0(x)=\int_{-\infty}^{x}(u_0-\bar
u)(y)\mathrm{d}y,\notag\\
&&\tilde{W}_{10}(x)=\int_{-\infty}^{x}\big(\frac{R}{\gamma-1}\theta_0+\frac12u_0^2-\frac{R}{\gamma-1}
\bar\theta-\frac12\bar u^2\big)(y)\mathrm{d}y,\label{initial16}
\end{eqnarray}
exist for any $x\in\mathbb{R}$, and satisfy
\begin{eqnarray}
(\Phi_0,\Psi_0,{W}_{10})(x)\in L^2(\mathbb{R}),\label{assumption12}
\end{eqnarray}
where $W_{10}=\frac{\gamma-1}{R}(\tilde{W}_{10}-\bar u\Psi_0)$. Note
that (\ref{assumption11}) and (\ref{assumption12}) imply that the initial perturbations have zero integrals:
\begin{eqnarray*}
&&\int_{-\infty}^{+\infty}(v_0-\bar
v)(x)\mathrm{d}x=0,~\int_{-\infty}^{+\infty}(u_0-\bar
u)(x)\mathrm{d}x=0,\notag\\
&&\int_{-\infty}^{+\infty}\big(\frac{R}{\gamma-1}\theta_0+\frac12u_0^2-\frac{R}{\gamma-1}
\bar\theta-\frac12\bar u^2\big)(x)\mathrm{d}x=0.
\end{eqnarray*}
The second result of this paper, i.e., the asymptotic stability of
shock profile $(\bar{v}, \bar{u}, \bar{\theta}, \bar{\phi})$ is
stated as follows.
\begin{thm}
Let $(v_\pm, u_\pm,\theta_\pm)$ satisfy (\ref{relation12}) and \eqref{quasi-neutral} holds, and $v_1\leq v_\pm\leq v_2,|u_\pm|\leq
u_1,\theta_1\leq \theta_\pm\leq \theta_2$. We also choose a positive
constant  $\varepsilon_0<\min\{v_1,\theta_1\}$. Let the shock
profile $(\bar v,\bar u,\bar\theta,\bar\phi)$ with the shock speed
$s> 0$ and with $v_+>v_->0$ be obtained as in Theorem 1.1. Assume
the initial data $(v_0,u_0,\theta_0)(x)$ satisfy
(\ref{assumption11}) and (\ref{assumption12}). Then there exist
positive constants $\delta_1(\leq\delta_0)$ and
$\varepsilon_1(\le\varepsilon_0)$ such that if
$|v_+-v_-|\le\delta_1$, and
\begin{eqnarray*}
\|(v_0-\bar v,u_0-\bar u,\theta_0-\bar\theta)\|_1+\|(\Phi_0, \Psi_0,
W_{10})\|\le\varepsilon_1
\end{eqnarray*}
then the problem (\ref{model12}) and (\ref{initial14}) has a unique
global solution $(v,u,\theta,\phi)$ satisfying
\begin{eqnarray*}
&&v-\bar v\in C([0,\infty);H^1(\mathbb{R}))\cap
L^2([0,\infty);H^1(\mathbb{R})),\\
&&u-\bar u,\theta-\bar\theta\in C([0,\infty);H^1(\mathbb{R}))\cap L^2([0,\infty);H^2(\mathbb{R}))\notag\\
&&\phi-\bar \phi\in C([0,\infty);H^2(\mathbb{R}))\cap L^2([0,\infty);H^2(\mathbb{R})),\notag\\
&&\inf_{t\geq0,x\in\mathbb{R}}\{v(t,x),\theta(t,x)\}>0.
\end{eqnarray*}
Moreover, the solution tends to the traveling wave solution in the
maximum norm:
\begin{eqnarray}
\sup_{x\in\mathbb{R}}|(v,u,\theta,\phi)(t,x)-(\bar v,\bar
u,\bar\theta,\bar\phi)(x-st)|\rightarrow0~{\rm
as}~t\rightarrow\infty.\label{inequality13}
\end{eqnarray}
\end{thm}

\begin{rmk}
In this article, we only study the existence and the asymptotic
stability of shock profiles of the full compressible NSP system
(\ref{model11}). We should mention that the corresponding initial
boundary value problem such as the out-flow problem and the inflow
problem for the one-dimensional and multi-dimensional full
compressible NSP system is surely more difficult, thus more
interesting. These are expected to be done in the forthcoming
papers.
\end{rmk}

Now we briefly state the ideas and main difficulties for the proof of Theorems 1.1 and 1.2. First, let us recall that, in the context of gas dynamics and kinetic theory, the equivalence of shock profiles in the Eulerian coordinates and Lagrangian coordinates has been shown in \cite{HWWY}. 
Then we only need to establish the existence of the traveling wave solution $(\bar{\rho}, \bar{u}, \bar{\theta}, \bar{\phi})(x-\bar{s}t)$, to the system (\ref{model11}). 
Moreover, we can rewrite the equations of the shock profile into a $4\times 4$ autonomous system of the first-order ordinary differential equations. 
Then, through a subtle analysis, it is found there is zero eigenvalue and two corresponding linear-independent eigenvectors of the Jacobi matrix of the linearized equations at the critical point $(\rho_-, \theta_-, z_-, w_-, \sqrt{R\gamma\theta_-+1})$. Therefore, the proof of Theorem 1.1 is finished by the center manifold theorem in \cite{MP}.

The stability analysis of the traveling wave is complicated by the incomplete parabolicity of the system and, more importantly, by the coupling with the self-consistent electric field. The former is a classical difficulty for the compressible Navier--Stokes equations: the equations for the velocity $u$ and the absolute temperature $\theta$ have a parabolic character, whereas the equation for the specific volume $v$ is hyperbolic, and this difficulty can be handled by the energy arguments developed in \cite{KM,KN,MN,MN1}. The electric field, however, changes both the construction of the integrated system and its energy structure. Using the Poisson equation, we first derive an integrated equation for the velocity perturbation as in \cite{DLZ}. For the third equation, the natural conserved quantity is the total energy containing the internal, kinetic, and electrostatic contributions; see \eqref{relation32}. And the resulting integrated equation is no longer directly parabolic in the temperature variable. To overcome this difficulty, we introduce $W$ in \eqref{relation35}, a carefully chosen integrated-energy variable that simultaneously preserves the underlying conservation structure, cancels the leading kinetic--electrostatic couplings, and recovers the heat-conductive principal part, so that $W_y$ represents the temperature perturbation at the leading order. This construction is different from that for the one-dimensional full Navier--Stokes equations in \cite{KM}.

A further key point in the basic energy estimate is to recover a coercive dissipation from the hyperbolic--parabolic coupling. For this purpose, we identify the following decomposition of the corresponding coefficient:
\[
\frac{R\bar\theta+1-\mu\bar u_y}{\bar v}
=
k_1(\bar v)+k_2(\bar v,\bar\phi_y).
\]
Here $k_1(\bar v)$ plays the role of the coefficient $k(\bar v)$ that generates the key positive energy structure in \cite{KM} and satisfies the same positivity properties, while $k_2(\bar v,\bar\phi_y)$ is an electrostatic correction with favorable decay along the traveling-wave profile. This decomposition allows the main energy framework of [23] to remain applicable. The remaining electrostatic coupling terms, in particular $F_{12}$, $F_{13}$, and $F_{22}$, require additional estimates. Thus, compared with the one-dimensional isentropic NSP system in \cite{DLZ} and the full Navier--Stokes system in \cite{KM}, the main additional issues lie in the construction of the integrated energy variable and the treatment of the electrostatic coupling.

\

The rest of this paper is organized as follows. In Section 2, we
shall prove Theorem 1.1 for the construction of the shock profile in
terms of (\ref{model11}) for the fixed viscosity coefficient $\mu >
0 $, heat-conductive $\kappa>0$ and the fixed Debye length $\lambda
> 0$. Next we reformulate the original problem in terms of the
perturbed functions and give the proof of Theorem 1.2 in Section~3.
We also collect the a priori estimates needed in the proof of our
main theorem, in Proposition 3.2. And Section~4 is used to
present the a-priori estimate. %Finally, we complete the proof of our
%main theorem in Section~5.

\noindent{\bf Notations.} Throughout this paper, $C$ denotes some
generic positive (generally large) constant and $c$ (generally
small) denotes two generic positive constants, where both $C$ and
$c$ may take different values in different places. $L^p(\mathbb{R})$
stands for the $L^p$-norm $(1 \leq p \leq +\infty)$, and for
simplicity we denote $\|\cdot\|_{L^2}$ by $\|\cdot\|$.  For a
non-negative integer $k$, $H^k(\mathbb{R})$ denotes the standard
Hilbert spaces of order $k$, and we write $\|\cdot\|_k$ for the
usual norm of $H^k(\mathbb{R})$. In addition, we denote by $C([0,
T]; H^k(\mathbb{R}))$ (resp. $L^2(0, T; H^k(\mathbb{R}))$) the space
of continuous (resp. square integrable) functions on $[0, T]$ taking
values in the space $H^k(\mathbb{R})$.

\section{Existence of shock profile with small amplitude}
In this section, we construct the shock profile solution $(\bar{v},
\bar{u}, \bar{\theta}, \bar{\phi})(x-st)$, to the system
(\ref{model12}), for the fixed constant viscosity $\mu>0$,
heat-conductive $\kappa>0$ and Debye length $\lambda>0$. First, let
us recall that, in the context of gas dynamics and kinetic theory,
the equivalence of shock profiles in the Eulerian coordinates and
Lagrangian coordinates has been shown in \cite{HWWY}. Then, the case
of the compressible NSP system is similar. Therefore, we only need
to consider the existence of the shock profile solution
$(\bar{\rho}, \bar{u}, \bar{\theta}, \bar{\phi})(x-\bar{s}t)$ with
the wave speed $\bar{s}$, to the system (\ref{model11}). For this,
let $(\rho,u,\theta,\phi)$ be a smooth traveling wave solution of
(\ref{model11}) which depends only on the variable $\xi=x-\bar{s}t$
with the wave speed $\bar{s}$ to be determined later and connects
the far-fields $(\rho_\pm, u_\pm, \theta_\pm,\phi_\pm)$ at
$x=\pm\infty$ respectively. The profile is thereby governed by the
following system of ODEs:
\begin{eqnarray}
\left\{
\begin{array}{lcr}
-\bar{s}\frac{\mathrm{d}\rho}{\mathrm{d}\xi}+\frac{\mathrm{d}(\rho
u)}{\mathrm{d}\xi}=0,\\
-\bar{s}\frac{\mathrm{d}(\rho
u)}{\mathrm{d}\xi}+\frac{\mathrm{d}(\rho
u^2+p)}{\mathrm{d}\xi}=\mu\frac{\mathrm{d}^2u}{\mathrm{d}\xi^2}-\rho\frac{\mathrm{d}\phi}{\mathrm{d}\xi},\\
-\bar{s}\frac{\mathrm{d}(\rho(e+\frac{1}{2}u^2))}{\mathrm{d}\xi}+\frac{\mathrm{d}(\rho
u(e+\frac{1}{2}u^2)+pu)}{\mathrm{d}\xi}=\kappa\frac{\mathrm{d}^2\theta}{\mathrm{d}\xi^2}
+\mu\frac{\mathrm{d}}{\mathrm{d}\xi}(u\frac{\mathrm{d}u}{\mathrm{d}\xi})-\rho
u\frac{\mathrm{d}\phi}{\mathrm{d}\xi},\\
-\lambda^2\frac{\mathrm{d}^2\phi}{\mathrm{d}\xi^2}=\rho-e^{\phi},
\end{array}\label{model21}\right.
\end{eqnarray}
with
\begin{eqnarray}
\lim\limits_{\xi\to\pm\infty}(\rho,u,\theta,\phi)(\xi)=(\rho_\pm,u_\pm,\theta_\pm,\phi_\pm).\label{initial21}
\end{eqnarray}

Similar to (\ref{model14}), we can rewrite  system (\ref{model21})
as
\begin{eqnarray}
\left\{
\begin{array}{lcr}
-\bar{s}\frac{\mathrm{d}\rho}{\mathrm{d}\xi}+\frac{\mathrm{d}(\rho
u)}{\mathrm{d}\xi}=0,\\
-\bar{s}\frac{\mathrm{d}(\rho
u)}{\mathrm{d}\xi}+\frac{\mathrm{d}(\rho
u^2+p)}{\mathrm{d}\xi}=\mu\frac{\mathrm{d}^2u}{\mathrm{d}\xi^2}
+\frac{\mathrm{d}}{\mathrm{d}\xi}\big(\frac{\lambda^2}{2}(\frac{\mathrm{d}\phi}{\mathrm{d}\xi})^2-e^\phi\big),\\
-\bar{s}\frac{\mathrm{d}}{\mathrm{d}\xi}\big(\rho(e+\frac{1}{2}u^2)+\rho\phi
+\frac{\lambda^2}{2}(\frac{\mathrm{d}\phi}{\mathrm{d}\xi})^2-e^\phi\big)+\frac{\mathrm{d}}{\mathrm{d}\xi}(\rho
u(e+\frac{1}{2}u^2)+pu+\rho
u\phi)\\
~~~~~~~~~~~~~=\kappa\frac{\mathrm{d}^2\theta}{\mathrm{d}\xi^2}
+\mu\frac{\mathrm{d}}{\mathrm{d}\xi}\big(u\frac{\mathrm{d}u}{\mathrm{d}\xi}\big),\\
-\lambda^2\frac{\mathrm{d}^2\phi}{\mathrm{d}\xi^2}=\rho-e^{\phi}.
\end{array}\right.\label{model22}
\end{eqnarray}
In this procedure, we have used the following facts:
\begin{eqnarray*}
-\rho\frac{\mathrm{d}\phi}{\mathrm{d}\xi}=
\frac{\lambda^2}{2}\frac{\mathrm{d}}{\mathrm{d}\xi}(\frac{\mathrm{d}\phi}{\mathrm{d}\xi})^2-\frac{\mathrm{d}}{\mathrm{d}\xi}e^\phi,
\end{eqnarray*}
and
\begin{eqnarray*}
-\rho
u\frac{\mathrm{d}\phi}{\mathrm{d}\xi}=-\frac{\mathrm{d}}{\mathrm{d}\xi}\rho
u\phi+\frac{\mathrm{d}}{\mathrm{d}\xi}\big(\bar{s}\rho\phi
+\frac{\lambda^2}{2}\bar{s}(\frac{\mathrm{d}\phi}{\mathrm{d}\xi})^2-\bar{s}e^\phi\big).
\end{eqnarray*}
Further, we obtain the R-H condition in the Eulerian coordinates as
follows:
\begin{eqnarray}
\left\{
\begin{array}{lcr}
-\bar{s}(\rho_+-\rho_-)+(\rho_+u_+-\rho_-u_-)=0,\\
-\bar{s}(\rho_+u_+-\rho_+u_+)+(\rho_+u^2_++R\rho_+\theta_+-\rho_-u^2_--R\rho_-\theta_-)+(\rho_+-\rho_-)=0,\\
-\bar{s}\big((\rho_+(\frac{R}{\gamma-1}\theta_++\frac{1}{2}u^2_+)+\rho_+\phi_+-e^{\phi_+})
-(\rho_-(\frac{R}{\gamma-1}\theta_-+\frac{1}{2}u^2_-)+\rho_-\phi_--e^{\phi_-})\big)\\
~~+\rho_+u_+\big
(\frac{R}{\gamma-1}\theta_++\frac{1}{2}u^2_++R\theta_++\phi_+\big)
-\rho_-u_-\big(\frac{R}{\gamma-1}\theta_-+\frac{1}{2}u^2_-+R\theta_-+\phi_-\big)=0,
\end{array}\right.\label{condition21}
\end{eqnarray}
and it is the same as the one for the following full quasi-neutral
type Euler system:
\begin{eqnarray}
\left\{
\begin{array}{lcr}
\rho_t+(\rho u)_x=0,\\
(\rho u)_t+(\rho u^2+R\rho\theta+\rho)_x=0,\\
\big(\rho(\frac{R}{\gamma-1}\theta+\frac{1}{2}u^2)+\rho
\ln\rho-\rho\big)_t+(\rho
u(\frac{R}{\gamma-1}\theta+\frac{1}{2}u^2)+R\rho\theta u+\rho
u\ln\rho)_x=0.
\end{array}\right.\label{model23}
\end{eqnarray}
which has the following eigenvalues:
\begin{eqnarray*}
\bar{\lambda}_1=u-\sqrt{R\gamma\theta+1},~ \bar{\lambda}_2=u,~
\bar{\lambda}_3=u+\sqrt{R\gamma\theta+1}.
\end{eqnarray*}

In the following let us construct the traveling wave with
compressibility, namely, $\rho_+<\rho_-$ and $\bar{s}<0$. That is,
we have the following existence of shock profile of (\ref{model11}).
\begin{thm}
For given data $(\rho_-,u_-,\theta_-)$ with $\rho_-,\theta_->0$,
there exist positive constants $\bar{\varepsilon}_0, c, C, b$ and
$C_k(k=0,1,2,\cdots)$ such that if $(\rho_+,u_+,\theta_+)$ satisfies
(\ref{condition21}) and $0<\rho_--\rho_+\le\bar{\delta}_0$, the
problem (\ref{model21}) or equivalently (\ref{model22}), has a
smooth traveling wave solution of the form $(\bar{\rho}, \bar{u},
\bar{\theta}, \bar{\phi})(x-\bar{s}t)$ connecting the far fields
$(\rho_\pm, u_\pm, \phi_\pm, \theta_\pm)$ with
$\phi_\pm=\ln\rho_\pm$, which is unique up to a spatial shift and
satisfies the following properties:
\begin{eqnarray}
\bar{\rho}_x(x-st)=\frac{\bar{\rho}^2\bar{u}_x(x-\bar{s}t)}{\rho_-(u_--s)},
c\bar{\rho}_x\le\bar{\phi}_x,\bar{\theta}_x\le
C\bar{\rho}_x<0\label{inequality21}
\end{eqnarray}
for any $x\in R$ and $t>0$, and
\begin{eqnarray}
\big|\frac{\mathrm{d}^k}{\mathrm{d}x^k}(\bar{\rho}-\rho_\pm,
\bar{u}-u_\pm, \bar{\theta}-\theta_\pm, \bar{\phi}-\phi_\pm)\big|
\le
C_k|\rho_+-\rho_-|^{k+1}e^{-b|\rho_+-\rho_-||x-st|}\label{inequality22}
\end{eqnarray}
for $x-st\lessgtr0$ and $k=0,1,2,\cdots$.
\end{thm}

\noindent{\bf Proof.} Without loss of generality, we can assume
$(\rho_-,u_-,\theta_-,\phi_-)=(1,0,1,0)$ in the following. First,
from the first equation of (\ref{model22}), one can solve $u$ by
$\rho$ as $u=\frac{\bar{s}(\rho-1)}{\rho}$. Replacing $u$ by $\rho$ in (\ref{model22}) and integrating
the resultant equation from $-\infty$ to $\xi$, one obtains the
following equivalent system
\begin{eqnarray} \left\{
\begin{array}{lcr}
\rho_\xi=\frac{\rho^2}{\mu
\bar{s}}((R\theta-\bar{s}^2)(\rho-1)+R(\theta-1)+\frac{\bar{s}^2}{\rho}(\rho-1)^2-\frac{\lambda^2}{2}w^2+z-1),\\
\theta_\xi=\frac{1}{\kappa}(-\frac{R\bar{s}}{\gamma-1}(\theta-1)-\bar{s}\log{z}+(z-1)\frac{\bar{s}}{\rho}
-\frac{\lambda^2}{2\rho}\bar{s}w^2+\frac{\bar{s}^3}{2\rho^2}(\rho-1)^2
+\frac{R\bar{s}(\rho-1)}{\rho}),\\
z_\xi=zw,\\
w_\xi=-\lambda^{-2}(\rho-z).
\end{array}\right.\label{model24}
\end{eqnarray}
with the far fields given by
\begin{eqnarray}
\lim_{\xi\rightarrow\pm\infty}(\rho,\theta,z,w)(\xi)=(\rho_{\pm},\theta_{\pm},\rho_{\pm},0).\label{initial23}
\end{eqnarray}
Here we have denoted $z=e^\phi$ and $w=\frac{\mathrm{d}\phi}{\mathrm{d}\xi}$. For convenience, we introduce 
$$U=(\rho, \theta, z, w),\quad  U_\pm=(\rho_\pm, \theta_\pm,z_\pm, w_\pm),$$ 
and write (\ref{model24}) as $\dot U=F(U)$, where $F(\cdot)$ denotes the vector field on the right hand of (\ref{model24}).

Once the autonomous ODE system (\ref{model24}) is derived, the rest of the proof follows the idea in \cite{MP}, and the subsequently similar arguments in \cite{DLZ}. We append $\dot\tau=0$ and consider the extended system 
\begin{eqnarray}
\left\{
\begin{array}{lcr}
\dot{\rho}=\frac{\rho^2}{\mu
\tau}((R\theta-\tau^2)(\rho-1)+R(\theta-1)+\frac{\tau^2}{\rho^2}(\rho-1)^2-\frac{\lambda^2}{2}w^2+z-1),\\
\dot{\theta}=\frac{1}{\kappa}(-\frac{R\tau}{\gamma-1}(\theta-1)-\tau\log{z}+(z-1)\frac{\tau}{\rho}-\frac{\lambda^2}{2\rho}\tau
w^2+\frac{\tau^3}{2\rho^2}(\rho-1)^2+\frac{R\tau(\rho-1)}{\rho}),\\
\dot{z}=zw,\\
\dot{w}=-\lambda^{-2}(\rho-z),\\
\dot{\tau}=0.
\end{array}\right.\label{model25}
\end{eqnarray}
It is obvious that $(\rho_\pm, \theta_\pm, z_\pm, w_\pm, \bar{s})$ are the only two critical points of (\ref{model25}).

Now we take $(\rho_-, \theta_-, z_-, w_-, \sqrt{R\gamma\theta_-+1})$$=(1,1,1,0,\sqrt{R\gamma+1})$ as a reference state and construct the center manifold of (\ref{model25}) around it.
Indeed, the Jacobian of
(\ref{model25}) at $=(1,1,1,0,\sqrt{R\gamma+1})$ is
\begin{eqnarray*}
J=
\begin{pmatrix}
\frac{1}{\mu  \tau_0 }(R- \tau_0 ^2) & \frac{R}{\mu  \tau_0 } &
\frac{1}{\mu
 \tau_0 } & 0 & 0 \\
\frac{R \tau_0 }{\kappa} & -\frac{R \tau_0 }{\kappa(\gamma-1)} & 0 & 0 & 0 \\
0 & 0 & 0 & 1 & 0 \\
-\lambda^{-2} & 0 & \lambda^{-2} & 0 & 0 \\
0 & 0 & 0 & 0 & 0
\end{pmatrix}, \quad \text{where~}\tau_0=\sqrt{R\gamma+1}.
\end{eqnarray*}
At the reference equilibrium $(1,1,1,0,\sqrt{R\gamma+1})$, the matrix $J$ has two linearly independent eigenvectors associated with the zero eigenvalue: $(1, \gamma-1, 1, 0, 0),~(0, 0, 0, 0, 1)$. 
Hence Proposition 3.2 of \cite{MP} yields a two-dimensional local center manifold $M^c$ invariant along the flow (\ref{model25}) of the form
\begin{eqnarray*}
U(\eta, \tau)=U_-+\eta(1, \gamma-1, 1, 0)+H(\eta, \tau),
\end{eqnarray*}
with $|\eta|+|\tau-\sqrt{R\gamma+1}|\le C$ for some constant $C> 0$.
Here $\tau$ is regarded as the parameter since $\dot{\tau}=0$ in the flow, and $
H(\eta, \tau)=(H^1, H^2, H^3,
H^4)(\eta, \tau)\in \mathbb{R}^4$
is the higher order term with
\begin{eqnarray}
H(0,\sqrt{R\gamma+1})=H_\eta(0,\sqrt{R\gamma+1})=H_\tau(0,\sqrt{R\gamma+1})=0.\label{relation23}
\end{eqnarray}
If $|\rho_+-1|$ is small enough, the equilibria $(U_-, \bar{s})$ and $(U_+, \bar{s})$ belong to $M^c$. Therefore, there exists $\eta_+ < 0$ such that $U_+=U_-+\eta_+(1, \gamma-1, 1, 0)+H(\eta_+, \bar{s})$, where $\eta_+<0$ follows from the compressibility. Restricting the flow to $\tau=\bar s$ gives the curve
\begin{eqnarray}
U(\eta, \bar{s})=(\rho(\eta), \theta(\eta), z(\eta),
w(\eta))=:U_-+\eta(1, \gamma-1, 1, 0)+H(\eta,
\bar{s})\label{relation24}
\end{eqnarray}
for $\eta\in(\eta_+, 0)$. 
$U(\eta, \bar{s})$ is invariant under the flow (\ref{model24}), and thus $\mathrm{d}U(\eta, \bar{s})/\mathrm{d}\eta$, the tangent vector of the curve, is parallel to $F(U(\eta, \bar{s}))$. Therefore,
the existence of the orbit of (\ref{model24}) connecting $U_-$ and $U_+$ is reduced to solve the ODE
\begin{eqnarray}
\left\{
\begin{array}{lcr}
\frac{\mathrm{d}\eta}{\mathrm{d}\xi}=g(\eta),  -\infty<\xi<+\infty, \\
\lim\limits_{\xi\to-\infty}\eta(\xi)=0,
\lim\limits_{\xi\to+\infty}\eta(\xi)=\eta_+
\end{array}\right.\label{relation25}
\end{eqnarray}
where $g(\eta)$ is determined by
\begin{eqnarray}
g(\eta)U_\eta(\eta, \bar{s})=F(U(\eta, \bar{s})).\label{relation26}
\end{eqnarray}

Now it suffices to prove the existence of the solution to
(\ref{relation25}). Note that since $\eta_+ < 0$, the standard ODE
theory shows that (\ref{relation25}) has a smooth solution if and
only if $g(\eta) < 0$ for all $\eta\in (\eta_+, 0)$. Since $U_+$ is
the unique state near $U_-$ such that $F(U_+) = 0$, it follows from
(\ref{relation26}) that $g(\eta)$ vanishes only at two end points of
$(\eta_+, 0)$. Moreover, 
$\dot{g}(0)$ is derived by differentiating (\ref{relation26}) and then taking
the inner product of the resulting equation with $U_\eta(0,
\bar{s})$ as
\begin{eqnarray}
\dot{g}(0)&=&\frac{U_\eta(0, \bar{s})\mathrm{d}F(U_-)U_\eta(0, \bar{s})^\top}{|U_\eta(0, \bar{s})|^2} \notag\\
&=&\frac{R\gamma+1-\bar{s}^2}{(\gamma+1)\mu
s}+O(H_\eta(0, \bar{s})) \notag\\
&=&\frac{\sqrt{2}(\gamma-1)\rho_+}{\mu(\gamma+1)\sqrt{1+\rho_+}\sqrt{R\gamma+1}}(1-\rho_+)+O((1-\rho_+)^2) %\notag\\
>0\label{relation27}
\end{eqnarray}
for $|\rho_+-1|$ small enough. Here in the last equality of
(\ref{relation27}), we have used (\ref{relation23}) and the
following facts:
\begin{eqnarray*}
\theta_+-1=\frac{\bar{s}^2}{R\rho^2_+}(\rho_+-1)-\frac{R+1}{R\rho_+}(\rho_+-1),
\end{eqnarray*}
and
\begin{eqnarray*}
\frac{1+\rho_+-\gamma(\rho_+-1)}{2(\gamma-1)\rho^2_+}\bar{s}^2=\frac{R\gamma+1}{(\gamma-1)\rho_+}
-\frac{1}{\rho_+}(\rho_+-1)+\frac{\rho_+-1-\ln{\rho_+}}{\rho_+-1},
\end{eqnarray*}
which can be derived by (\ref{condition21}). Therefore, it holds
that $g(\eta) < 0$ in $(\eta_+, 0)$, and thus (\ref{relation25})
admits a smooth solution $\eta(\xi)$ for all $\xi\in \mathbb{R}$.
By setting
\begin{eqnarray*}
(\bar{\rho}, \bar{u}, \bar{\theta}, \bar{\phi})(\xi):= (\rho,
\bar{s}(1-\rho^{-1}), \theta, \log{z})(\eta(\xi))
\end{eqnarray*}
where $\rho(\cdot),\theta(\cdot)$ and $z(\cdot)$ are defined in
(\ref{relation24}), it yields that
$(\bar{\rho},\bar{u},\bar{\theta},\bar{\phi})$ solves
(\ref{model21}) with (\ref{initial21}).

The monotonicity
(\ref{inequality21}) follows from the representation of the center manifold and the sign of $d\eta/d\xi$:
\begin{eqnarray*}
&&\frac{\mathrm{d}\bar{\rho}}{\mathrm{d}\xi}=\bar{\rho}(\eta(\xi))\frac{\mathrm{d}\eta}{\mathrm{d}\xi}
=(1+H^1_\eta(\eta(\xi), \bar{s}))\frac{\mathrm{d}\eta}{\mathrm{d}\xi}\sim\frac{\mathrm{d}\eta}{\mathrm{d}\xi}<0, \\
&&\frac{\mathrm{d}\bar{u}}{\mathrm{d}\xi}=\frac{\bar{s}}{\bar{\rho}^2}\frac{\mathrm{d}\bar{\rho}}{\mathrm{d}\xi}<0,\\
&&\frac{\mathrm{d}\bar{\theta}}{\mathrm{d}\xi}=\bar{\theta}(\eta(\xi))\frac{\mathrm{d}\eta}{\mathrm{d}\xi}
=(\gamma-1+H^2_\eta(\eta(\xi), \bar{s}))\frac{\mathrm{d}\eta}{\mathrm{d}\xi}\sim\frac{\mathrm{d}\bar\rho}{\mathrm{d}\xi}<0, \\
&&\frac{\mathrm{d}\bar{\phi}}{\mathrm{d}\xi}=z^{-1}\dot{z}(\eta(\xi))\frac{\mathrm{d}\eta}{\mathrm{d}\xi}
=z^{-1}(1+H^3_\eta(\eta(\xi),
\bar{s}))\frac{\mathrm{d}\eta}{\mathrm{d}\xi}\sim\frac{\mathrm{d}\bar{\rho}}{\mathrm{d}\xi}<0,
\end{eqnarray*}
provided that $|\rho_+-1|$ is sufficiently small. Furthermore, to
verify (\ref{inequality22}) for $k = 0$, it follows from
(\ref{relation27}) that
\begin{eqnarray*}
|\eta(\xi)|\le C|\eta_+|e^{-b(1-\rho_+)|\xi|}\le
C(1-\rho_+)e^{-b(1-\rho_+)|\xi|}, \qquad \xi<0,
\end{eqnarray*}
with some constants $C > 0$ and $b> 0$ independent of $\xi$. Then it
follows from (\ref{relation24}) that
\begin{eqnarray*}
(\bar{\rho}-1, \bar{u}, \bar{\theta}-1, \bar{\phi})(\xi)\le
C(1-\rho_+)e^{-b(1-\rho_+)|\xi|}, \qquad \xi<0.
\end{eqnarray*}
Similarly, one also has
\begin{eqnarray*}
(\bar{\rho}-1, \bar{u}, \bar{\theta}-1, \bar{\phi})(\xi)\le
C(1-\rho_+)e^{-b(1-\rho_+)|\xi|}, \qquad \xi>0.
\end{eqnarray*}
Hence, (\ref{inequality22}) for $k = 0$ is proved. Estimates on (\ref{inequality22}) with $k\geq1$ for those high order derivatives of $(\bar{\rho}, \bar{u}, \bar{\theta}, \bar{\phi})$ can be similarly obtained by differentiating (\ref{model21}), and the details of the proof are omitted for brevity. It is also straightforward to show the uniqueness (up to a shift) for the ODE system (\ref{model21}) and (\ref{initial21}). Therefore, we complete the proof of Theorem 2.1.

\section{Reformulation of the original problem}
In this section, we reformulate (\ref{model12})
and (\ref{initial14}) around the traveling wave $(\bar{v}, \bar{u}, \bar{\theta},
\bar{\phi})(x-st)$ and prepare the proof of Theorem 1.2. In the moving coordinate $y=x-st$, set
\[
(\varphi,\psi,w,\sigma)=(v,u,\theta,\phi)-(\bar v,\bar u,\bar\theta,\bar\phi).
\]
Subtracting the profile equations gives the following system
\begin{eqnarray}
\left\{
\begin{array}{lcr}
\varphi_t-s\varphi_y-\psi_y=0, \\
\psi_t-s\psi_y+R\big(\frac{\bar\theta+w}{\bar{v}+\varphi}-\frac{\bar{\theta}}{\bar{v}}\big)_y=\mu\big(\frac{\psi_y}{
\bar{v}+\varphi}\big)_y +\mu\big(\frac{\bar
u_y}{\bar{v}+\varphi}-\frac{\bar{u}_y}{\bar{v}}\big)_y-\big(\frac{\bar \phi_y+\sigma_y}{\bar{v}+\varphi}-\frac{\bar{\phi}_y}{\bar{v}}\big), \\
\big(\frac{R}{\gamma-1}w+\frac{1}{2}\psi^2+\bar{u}\psi\big)_t
-s\big(\frac{R}{\gamma-1}w+\frac{1}{2}\psi^2+\bar{u}\psi\big)_y
+R\big(\frac{\bar\theta+w}{\bar{v}+\varphi}\psi+(\frac{\bar\theta+w}{\bar{v}+\varphi}-\frac{\bar{\theta}}{\bar{v}})\bar{u}\big)_y\\
~~=\big(\frac{\kappa}{\bar{v}+\varphi}w_y+(\frac{\kappa}{\bar{v}+\varphi}-\frac{\kappa}{\bar{v}})\bar{\theta}_y\big)_y+\big(\frac{\mu}{
\bar{v}+\varphi}(\psi\psi_y+\bar u\psi_y+\bar u_y\psi)+(\frac{\mu}
{\bar{v}+\varphi}-\frac{\mu}{\bar
v})\bar u\bar u_y\big)_y\\
~~~~~~-\big(\frac{\bar u+\psi}{\bar{v}+\varphi}(\bar\phi_y+\sigma_y)-\frac{\bar u}{\bar v}\bar\phi_y\big), \\
-\lambda^2\big(\frac{\bar{\phi}_y+\sigma_y}{\bar{v}+\varphi}-\frac{\bar{\phi}_y}{\bar{v}}\big)_y
=\bar{v}e^{\bar{\phi}}(1-e^\sigma)-\varphi e^{\bar{\phi}}e^\sigma,
\end{array}\right.\label{model30}
\end{eqnarray}
and
\begin{eqnarray}
(\varphi, \psi, w)(0, y)=(\varphi_0, \psi_0, w_0)(y)=(v_0-\bar{v},
u_0-\bar{u}, \theta_0-\bar{\theta})(y).\label{initial31}
\end{eqnarray}
(\ref{model30})$_3$ is equivalent to
\begin{eqnarray}
&&\frac{R}{\gamma-1}(w_t-sw_y)+\frac{R(\bar\theta+w)}{\bar{v}+\varphi}\psi_y
+R\big(\frac{\bar\theta+w}{\bar{v}+\varphi}-\frac{\bar{\theta}}{\bar{v}}\big)\bar{u}_y\notag\\
&=&\big(\frac{\kappa}{\bar{v}+\varphi}w_y+(\frac{\kappa}{\bar{v}+\varphi}-\frac{\kappa}{\bar{v}})\bar{\theta}_y\big)_y
+\frac{\mu}{\bar{v}+\varphi}(\psi^2_y+2\bar{u}_y\psi_y)+\big(\frac{\mu}{\bar{v}+\varphi}-\frac{\mu}{\bar{v}}\big)\bar{u}^2_y.\label{equality30}
\end{eqnarray}
The momentum equation (\ref{model30})$_2$ and the temperature equation (3.3) have a parabolic form but are not in conservation form. 
Using the Poisson equation, the electrostatic force and work terms can be rewritten by the following identities.
\begin{eqnarray*}
-\big(\frac{\bar\phi_y+\sigma_y}{\bar{v}+\varphi}-\frac{\bar{\phi}_y}{\bar{v}}\big)
&=&-\big(\frac{1}{\bar{v}+\varphi}-\frac{1}{\bar{v}}\big)_y
+\frac{\lambda^2}{2}\big((\frac{\bar\phi_y+\sigma_y}{\bar{v}+\varphi})^2-(\frac{\bar{\phi}_y}{\bar{v}}\big)^2)_y\\
&&-\lambda^2\big(\frac{1}{\bar{v}+\varphi}(\frac{\bar\phi_y+\sigma_y}{\bar{v}+\varphi})_y
-\frac{1}{\bar{v}}(\frac{\bar{\phi}_y}{\bar{v}})_y\big)_y,
\end{eqnarray*}
and
\begin{eqnarray*}
&&-\big(\frac{\bar
u+\psi}{\bar{v}+\varphi}(\bar\phi_y+\sigma_y)-\frac{\bar u}{\bar
v}\bar\phi_y\big)=\big((\bar{v}+\varphi)e^{\bar\phi+\sigma}-\sigma
+\frac{\lambda^2(\bar\phi_y+\sigma_y)^2}{2(\bar{v}+\varphi)}\big)_t
-s\big((\bar{v}+\varphi)e^{\bar\phi+\sigma}-\bar{v}e^{\bar{\phi}}\notag\\
&-&\sigma
+\frac{\lambda^2(\bar\phi_y+\sigma_y)^2}{2(\bar{v}+\varphi)}-\frac{\lambda^2\bar{\phi}^2_y}{2\bar{v}}\big)_y
-\lambda^2\big(\frac{(\bar\phi_y+\sigma_y)\sigma_t}{\bar{v}+\varphi}\big)_y+\big(\frac{\lambda^2}{2}(\bar
u+\psi)(\frac{\bar\phi_y+\sigma_y}{\bar{v}+\varphi})^2\\
&-&\frac{\lambda^2}{2}\bar{u}(\frac{\bar{\phi}_y}{\bar{v}})^2-(\bar
u+\psi)e^{\bar\phi+\sigma}+\bar{u}e^{\bar{\phi}}\big)_y
+\lambda^2s\big(\frac{(\bar\phi_y+\sigma_y)^2}{\bar{v}+\varphi}-\frac{\bar{\phi}_y^2}{\bar{v}}\big)_y.
\end{eqnarray*}
Substituting these identities into the momentum and energy equations gives the conservation forms:
\begin{eqnarray}
&&\psi_t-s\psi_y+\big(\frac{R(\bar\theta+w)+1}{\bar{v}+\varphi}
-\frac{R\bar{\theta}+1}{\bar{v}}\big)_y=\mu\big(\frac{\psi_y}{
\bar{v}+\varphi}\big)_y +\mu\big(\frac{\bar
u_y}{\bar{v}+\varphi}-\frac{\bar{u}_y}{\bar{v}}\big)_y\notag\\
&+&\frac{\lambda^2}{2}\big((\frac{\bar{\phi}_y+\sigma_y}{\bar{v}+\varphi})^2
-(\frac{\bar{\phi}_y}{\bar{v}})^2\big)_y
-\lambda^2\big(\frac{1}{\bar{v}+\varphi}(\frac{\bar\phi_y+\sigma_y}{\bar{v}+\varphi})_y
-\frac{1}{\bar{v}}(\frac{\bar{\phi}_y}{\bar{v}})_y\big)_y,\label{relation31}
\end{eqnarray}
and
\begin{eqnarray}
&&\Big(\frac{R}{\gamma-1}w+\frac{1}{2}\psi^2+\bar{u}\psi+\sigma-(\bar{v}+\varphi)e^{\bar\phi+\sigma}
+\bar{v}e^{\bar{\phi}}-\frac{\lambda^2}{2}\frac{(\bar\phi_y+\sigma_y)^2}{\bar{v}+\varphi}+
\frac{\lambda^2}{2}\frac{\bar{\phi}^2_y}{\bar{v}}\Big)_t\notag\\
&&-s\Big(\frac{R}{\gamma-1}w+\frac{1}{2}\psi^2+\bar{u}\psi+\sigma-(\bar{v}+\varphi)e^{\bar\phi+\sigma}
+\bar{v}e^{\bar{\phi}}-\frac{\lambda^2}{2}\frac{(\bar\phi_y+\sigma_y)^2}{\bar{v}+\varphi}+
\frac{\lambda^2}{2}\frac{\bar{\phi}^2_y}{\bar{v}}\Big)_y\notag\\
&&+R\big(\frac{\bar\theta+w}{\bar{v}+\varphi}\psi+(\frac{\bar\theta+w}{\bar{v}+\varphi}
-\frac{\bar{\theta}}{\bar{v}})\bar{u}\big)_y
+\lambda^2\big(\frac{\bar\phi_y+\sigma_y}{\bar{v}+\varphi}(\sigma_t-s(\bar\phi_y+\sigma_y))
+\frac{s}{\bar{v}}\bar{\phi}^2_y\big)_y\notag\\
&&+\big((\bar
u+\psi)e^{\bar\phi+\sigma}-\bar{u}e^{\bar{\phi}}\big)_y
-\frac{\lambda^2}{2}\big((\bar
u+\psi)(\frac{\bar\phi_y+\sigma_y}{\bar{v}+\varphi})^2
-\bar{u}(\frac{\bar{\phi}_y}{\bar{v}})^2\big)_y\notag\\
&=&\big(\frac{\kappa}{\bar{v}+\varphi}w_y+(\frac{\kappa}{\bar{v}+\varphi}
-\frac{\kappa}{\bar{v}})\bar{\theta}_y\big)_y+\big(\frac{\mu}{
\bar{v}+\varphi}(\psi\psi_y+\bar u\psi_y+\bar u_y\psi)+(\frac{\mu}
{\bar{v}+\varphi}-\frac{\mu}{\bar v})\bar u\bar
u_y\big)_y.\label{relation32}
\end{eqnarray}

Furthermore, by setting
\begin{eqnarray}
\tilde{W}_{20}(x)=\int_{-\infty}^x\big(\phi_0-\bar{\phi}-v_0e^{\phi_0}+\bar{v}e^{\bar{\phi}}
-\frac{\lambda^2}{2}\frac{\phi^2_{0y}}{v_0}+\frac{\lambda^2}{2}\frac{\bar{\phi}^2_y}{\bar{v}}\big)\mathrm{d}y, \label{relation33}
\end{eqnarray}
together with (\ref{initial16}) and
(\ref{assumption12}), we have $\tilde{W}_{20}(x)\in
L^2(\mathbb{R})$, see Remark 4.2 below. Denote
$\tilde{W}_{0}:=\tilde{W}_{10}+\tilde{W}_{20}$, in which
$\tilde{W}_{10}$ is defined in \eqref{initial16}, and
\begin{eqnarray}
\tilde{w}_0=\frac{R}{\gamma-1}w_0+\frac{1}{2}\psi^2_0+\bar{u}\psi_0+\phi_0-\bar{\phi}-v_0e^{\phi_0}+\bar{v}e^{\bar{\phi}}
-\frac{\lambda^2}{2}\frac{\phi^2_{0y}}{v_0}+\frac{\lambda^2}{2}\frac{\bar{\phi}^2_y}{\bar{v}}.\label{relation34}
\end{eqnarray}
Then from (\ref{relation34}), we have $(\varphi_0, \psi_0,
\tilde{w}_0)=(\Phi_{0y}, \Psi_{0y}, \tilde{W}_{0y})$. Taking it into
account, we seek the solution of the problem (\ref{model30}) and
(\ref{initial31}) in the form
\begin{eqnarray}
(\varphi, \psi, \tilde{w})=(\Phi_{y}, \Psi_{y},
\tilde{W}_{y}),\label{relation30}
\end{eqnarray}
where
\begin{eqnarray*}
\tilde{w}=\frac{R}{\gamma-1}w+\frac{1}{2}\psi^2+\bar{u}\psi+\sigma-(\bar{v}+\varphi)e^{\bar{\phi}+\sigma}+\bar{v}e^{\bar{\phi}}
-\frac{\lambda^2}{2}\frac{(\bar{\phi}_y+\sigma_y)^2}{\bar{v}+\varphi}+\frac{\lambda^2}{2}\frac{\bar{\phi}^2_y}{\bar{v}}.
\end{eqnarray*}
Let us introduce the new variable
\begin{eqnarray}
W=\frac{\gamma-1}{R}(\tilde{W}-\bar{u}\Psi+e^{\phi_\pm}\Phi),\label{relation35}
\end{eqnarray}
and thus $w$ can be written in the form
\begin{eqnarray}
w&=&W_y+\frac{\gamma-1}{R}\big(\bar{u}_y\Psi-\frac{1}{2}\psi^2-\sigma-e^{\phi_\pm}\Phi_y
+(\bar{v}+\Phi_y)e^{\bar{\phi}+\sigma}-\bar{v}e^{\bar{\phi}}\notag\\
&&+\frac{\lambda^2}{2}\frac{(\bar{\phi}_y+\sigma_y)^2}{\bar{v}+\Phi_y}
-\frac{\lambda^2}{2}\frac{\bar{\phi}^2_y}{\bar{v}}\big).\label{relation36}
\end{eqnarray}
Substituting (\ref{relation36}) into (\ref{model30})$_1$,
(\ref{relation31}) and  (\ref{relation32}), and integrating once
with respect to $y$, we get the system for $(\Phi, \Psi, W, \sigma)$
\begin{eqnarray}
\left\{
\begin{array}{lcr}
\Phi_t-s\Phi_y-\Psi_y=0,\\
\Psi_{t}-s\Psi_y+(\frac{R(\bar\theta+w)+1}{\bar{v}+\Phi_y}-\frac{R\bar{\theta}+1}{\bar{v}})
=\frac{\mu}{\bar{v}+\Phi_y}\Psi_{yy}+\big(\frac{\mu}{\bar{v}+\Phi_y}-\frac{\mu}{\bar{v}}\big)\bar{u}_y\\
~~~~+\frac{\lambda^2}{2}\big((\frac{\bar{\phi}_y+\sigma_y}{\bar{v}+\Phi_y})^2-(\frac{\bar{\phi}_y}{\bar{v}})^2\big)
-\lambda^2\big(\frac{1}{\bar{v}+\Phi_y}(\frac{\bar\phi_y+\sigma_y}{\bar{v}+\Phi_y})_y-\frac{1}{\bar{v}}(\frac{\bar{\phi}_y}{\bar{v}})_y\big),\\
\frac{R}{\gamma-1}(W_t-sW_y)-s\bar{u}_y\Psi+(\frac{R(\bar\theta+w)+1}{\bar{v}+\Phi_y}-e^{\phi_\pm})\Psi_y
=\frac{\kappa}{\bar{v}+\Phi_y}w_y+\kappa(\frac{1}{\bar{v}+\Phi_y}-\frac{1}{\bar{v}})\bar{\theta}_y\\
~~~~~~~~~~+\frac{\mu}{\bar{v}+\Phi_y}(\bar{u}_y\Psi_y+\Psi_y\Psi_{yy})
-\lambda^2(\frac{\bar{\phi}_y+\sigma_y}{\bar{v}+\Phi_y}\sigma_t-\frac{s(\bar{\phi}_y+\sigma_y)^2}{\bar{v}+\Phi_y}
+\frac{s\bar{\phi}^2_y}{\bar{v}})\\
~~~~~~~~~~+\frac{\lambda^2}{2}\Psi_y\big(\frac{\bar{\phi}_y+\sigma_y}{\bar{v}+\Phi_y}\big)^2
-\lambda^2\frac{\Psi_y}{\bar{v}+\Phi_y}
\big(\frac{\bar{\phi}_y+\sigma_y}{\bar{v}+\Phi_y}\big)_y,\\
-\lambda^2\Big(\frac{\bar{\phi}_y+\sigma_y}{v}-\frac{\bar{\phi}_y}{\bar{v}}\Big)_y
=-\bar{v}e^{\bar{\phi}}\sigma+\bar{v}e^{\bar{\phi}}(1+\sigma-e^\sigma)-\Phi_y
e^{\bar{\phi}}+\Phi_y e^{\bar{\phi}}(1-e^\sigma),
\end{array}\right.\label{model31}
\end{eqnarray}
where we used the fact:
\begin{eqnarray*}
(\bar{u}+\Psi_y)e^{\bar{\phi}+\sigma}-\bar{u}e^{\bar{\phi}}=\frac{1}{\bar{v}+\Phi_y}\Psi_y
+(\frac{1}{\bar{v}+\Phi_y}-\frac{1}{\bar{v}})\bar{u}+\lambda^2\big(\frac{\bar{u}+\Psi_y}{\bar{v}+\Phi_y}
(\frac{\bar{\phi}_y+\sigma_y}{\bar{v}+\Phi_y})_y-\frac{\bar{u}}{\bar{v}}(\frac{\bar{\phi}_y}{\bar{v}})_y\big).
\end{eqnarray*}
The initial data for (\ref{model31}) is
\begin{eqnarray}
(\Phi,\Psi,W)(0,y)=(\Phi_0,\Psi_0,W_0)(y).\label{initial32}
\end{eqnarray}

Furthermore, we can rewrite the system (\ref{model31}) in the form
\begin{eqnarray}
\left\{
\begin{array}{lcr}
\Phi_t-s\Phi_y-\Psi_y=0, \\
\Psi_t-s\Psi_y-\frac{1}{\bar{v}}k_1(\bar{v})\Phi_y-\frac{k_2(\bar{v},
\bar{\phi}_y)}{\bar{v}}\Phi_y
+\frac{R}{\bar{v}}W_y=\frac{\mu}{\bar{v}}\Psi_{yy}-\frac{\gamma-1}{\bar{v}}\bar{u}_y\Psi+F_1, \\
\frac{R}{\gamma-1}(W_t-sW_y)-s\bar{u}_y\Psi+(k_1(\bar{v})-e^{\phi_\pm})\Psi_y+k_2(\bar{v},
\bar{\phi}_y)\Psi_y
-\frac{\kappa}{\bar{v}}(W_y+\frac{\gamma-1}{R}\bar{u}_y\Psi)_y \\
~~~~~~~~~~+\frac{\kappa}{\bar{v}^2}\bar{\theta}_y\Phi_y=F_2,\\
-\lambda^2\big(\frac{\bar{\phi}_y+\sigma_y}{v}-\frac{\bar{\phi}_y}{\bar{v}}\big)_y
=-\bar{v}e^{\bar{\phi}}\sigma-\Phi_y e^{\bar{\phi}}+F_3,
\end{array}\right.\label{model32}
\end{eqnarray}
where
\begin{eqnarray*}
F_1&=&\frac{\gamma-1}{2\bar{v}}\Psi^2_y-\frac{\Phi_y}{\bar{v}(\bar{v}+\Phi_y)}(\mu\Psi_{yy}+k_1(\bar{v})\Phi_y+k_2(\bar{v},
\bar{\phi}_y)\Phi_y-Rw)+\frac{\lambda^2}{2}\big((\frac{\bar{\phi}_y+\sigma_y}{\bar{v}+\Phi_y})^2-(\frac{\bar{\phi}_y}{\bar{v}})^2\big)\\
&&+\frac{\gamma-1}{\bar{v}}\big(\sigma-(\bar{v}+\Phi_y)e^{\bar{\phi}+\sigma}+\bar{v}e^{\bar{\phi}}
+e^{\phi_\pm}\Phi_y-\frac{\lambda^2}{2}(\frac{(\bar{\phi}_y+\sigma_y)^2}{\bar{v}+\Phi_y}-\frac{\bar{\phi}^2_y}{\bar{v}})\big) \\
&&
-\lambda^2\big(\frac{1}{\bar{v}+\Phi_y}(\frac{\bar{\phi}_y+\sigma_y}{\bar{v}+\Phi_y})_y-\frac{1}{\bar{v}}
(\frac{\bar{\phi}_y}{\bar{v}})_y\big) \\
&=:&F_{11}+F_{12}+F_{13},\\
F_2&=&-\frac{\kappa(\gamma-1)}{R\bar{v}}\Psi_y\Psi_{yy}+\frac{\Psi_y}{\bar{v}+\Phi_y}(k_1(\bar{v})\Phi_y+k_2(\bar{v},
\bar{\phi}_y)\Phi_y
+\mu\Psi_{yy}-Rw) \\
&&-\frac{\kappa\Phi_y}{\bar{v}(\bar{v}+\Phi_y)}(w_y-\frac{1}{\bar{v}}\bar{\theta}_y\Phi_y)
+\frac{\lambda^2}{2}\Psi_y\big(\frac{\bar{\phi}_y+\sigma_y}{\bar{v}+\Phi_y}\big)^2
-\lambda^2\frac{\Psi_y}{\bar{v}+\Phi_y}\big(\frac{\bar{\phi}_y+\sigma_y}{\bar{v}+\Phi_y}\big)_y\\
&&-\lambda^2\big(\frac{\bar{\phi}_y+\sigma_y}{\bar{v}+\Phi_y}\sigma_t-s\frac{(\bar{\phi}_y+\sigma_y)^2}{\bar{v}+\Phi_y}
+\frac{s\bar{\phi}^2_y}{\bar{v}^2}\big)-\frac{\kappa(\gamma-1)}{R\bar{v}}
\big(\sigma-(\bar{v}+\Phi_y)e^{\bar{\phi}+\sigma} \\
&&+\bar{v}e^{\bar{\phi}}+e^{\phi_\pm}\Phi_y
-\frac{\lambda^2}{2}(\frac{(\bar{\phi}_y+\sigma_y)^2}{\bar{v}+\Phi_y}-\frac{\bar{\phi}^2_y}{\bar{v}})\big)_y \\
&=:&F_{21}+F_{22},
\end{eqnarray*}
and
\begin{eqnarray*}
F_3=\bar{v}e^{\bar{\phi}}(1+\sigma-e^\sigma)+\Phi_y
e^{\bar{\phi}}(1-e^\sigma).
\end{eqnarray*}
In the above derivation, we have used the equality:
\begin{eqnarray*}
\frac{R\bar{\theta}+1-\mu\bar{u}_y}{\bar{v}}=k_1(\bar{v})+k_2(\bar{v},
\bar{\phi}_y)
\end{eqnarray*}
which is derived by (\ref{model13}) or (\ref{model14}), and
(\ref{initial15}), here
\begin{eqnarray}
k_1(\bar{v}):=\frac{R\theta_\pm+1}{v_\pm}-s^2(\bar{v}-v_\pm),\quad k_2(\bar{v},
\bar{\phi}_y):=\frac{\lambda^2}{2}(\frac{\bar{\phi}_y}{\bar{v}})^2-\frac{\lambda^2}{\bar{v}}
(\frac{\bar{\phi}_y}{\bar{v}})_y. \label{def-k}
\end{eqnarray}

Then we have the following Theorem.
\begin{thm}
Under the assumptions of Theorem 1.2, the problem (\ref{model31})
and (\ref{initial32}) has a unique global solution $(\Phi, \Psi,
W,\sigma)$ satisfying $\Phi, \Psi, W,\sigma\in C([0, +\infty);
H^2(\mathbb{R})), \Phi_y\in L^2([0, +\infty)$; $H^1(\mathbb{R})),
\Psi_y, W_y,\sigma\in L^2([0, +\infty); H^2(\mathbb{R}))$ and
\begin{eqnarray}
&&\|(\Phi, \Psi,W)(t)\|^2 +\|(\varphi, \psi,
w)(t)\|^2_1+\|\sigma\|_2^2\notag\\
&+&\int_{0}^{t}(\|(\bar{v}^{\frac{1}{2}}_y\Psi,
\bar{v}^{\frac{1}{2}}_yW)(\tau)\|^2 +\|\varphi(\tau)\|^2_1+\|(\psi,
w,\sigma)\|^2_2)\mathrm{d}\tau \leq C\label{inequality31}
\end{eqnarray}
for $t\in[0, \infty)$. %where $(\varphi, \psi, w)$ is defined in
%(\ref{relation30}) and (\ref{relation36}), and $(\varphi_0, \psi_0,
%w_0)$ is the initial data corresponding to $(\varphi, \psi, w)$ (see
%(\ref{initial31})).
\end{thm}

\vskip0.2cm \noindent{\bf The proof of Theorem 1.2.} Form the
relation  $(\varphi, \psi, w, \sigma)=(v, u, \theta,
\phi)(t,y)-(\bar{v}, \bar{u}, \bar{\theta}, \bar{\phi})(y)$ and
Theorem 3.1, we can obtain the global existence of the solution
$(v,u,\theta,\phi)$ of the Cauchy problem (\ref{model12}) and
(\ref{initial14}). In the following, we need to prove
(\ref{inequality13}). In fact, from (\ref{model30}),
(\ref{relation30}), (\ref{relation35}), (\ref{relation36}) and
(\ref{inequality31}), we can obtain
\begin{eqnarray*}
 \int_0^\infty  (\|\varphi_y\|^2 + \|\psi_y\|^2+
 \|w_y\|^2+\|\sigma_y\|^2)\mathrm{d}t <+\infty,
\end{eqnarray*}
and
\begin{eqnarray*}
 \int_0^\infty |\frac{\mathrm{d}}{\mathrm{d}t}(\|\varphi_y\|^2 + \|\psi_y\|^2+
 \|w_y\|^2+\|\sigma_y\|^2)|\mathrm{d}t <+\infty,
\end{eqnarray*}
which easily leads to
\begin{eqnarray}
 \lim_{t\to +\infty}\|(\varphi_y,\psi_y,w_y,\sigma_y)(t)\|=0.
 \label{inequality32}
\end{eqnarray}
Then using Sobolev inequality, (\ref{inequality32}) and $(\varphi,
\psi, w, \sigma)=(v, u, \theta, \phi)(t,y)-(\bar{v}, \bar{u},
\bar{\theta}, \bar{\phi})(y)$, we immediately obtain the large time
behavior of the solution, that is, (\ref{inequality13}). This
completes the proof of Theorem 1.2.

\

In the following, it is sufficient to prove Theorem 3.1. The local
existence of a small solution of (\ref{model31}) and
(\ref{initial32}) can be shown in the standard way (cf.
\cite{KN,TYZZ,ZLZ}, etc.). So, for the proof of Theorem 3.1, it is
suffice to show the following a priori estimates, because the
continuation argument used in the proof of Theorem 3.1 is also
applicable to this case.

\begin{prop}
(a priori estimate) Let $(\Phi, \Psi, W,\sigma)$ (for some $T>0$ )
be a solution of the problem (\ref{model31}) and (\ref{initial32}).
Then there exist positive constants $\delta_2(\leq\delta_1),
\varepsilon_2(\le\varepsilon_1)$ and $C$, which are independent of
$T$ and $(v_\pm, u_\pm, \theta_\pm, \phi_\pm)$ such that if
$|v_+-v_-|\le\delta_2$, and
\begin{eqnarray}
E(t)=\sup_{0\le\tau\le t}(\|(\Phi, \Psi, W)(\tau)\|+\|(\varphi,
\psi,
w)(\tau)\|^2_1+\|\sigma\|_2^2)\le\varepsilon_2\label{assumption31}
\end{eqnarray}
for any $t\in[0,T]$, then the following estimate holds for $t\in[0,
T]$,
\begin{eqnarray}
E(t)^2+\int_{0}^{t}(\|\bar{v}^{\frac{1}{2}}_y(\Psi,
W,\sigma)(\tau)\|^2 +\|\varphi(\tau)\|^2_1+\|(\psi,
w,\sigma)\|^2_2)d\tau\le CE^2_0.\label{inequality33}
\end{eqnarray}
\end{prop}

\section{A priori estimate}
In this section, we prove the {\it a priori} estimates for the unknown function $(\Phi,\Psi,W,\sigma)(t,x)$ in Proposition 3.2. 
Let  $(\Phi,\Psi,W,\sigma)(t,x)$ be the classical solution  to the problem (\ref{model31}) and (\ref{initial32}) on $[0,T]$ with the regularity stated below.
\begin{eqnarray*}
(\Phi, \Psi, W,\sigma)\in C(0, T; H^2); \Phi_y\in L^2(0, T; H^1);
\Psi_y, W_y,\sigma_y\in L^2(0, T; H^2).
\end{eqnarray*}
Then from Theorem 1.1, the Sobolev inequality and
(\ref{assumption31}), it yields easily that there exist two
positive constants $c$ and $C$ such that
\begin{eqnarray}
    0<c\leq \bar v,v,\bar\theta,\theta\leq C,~|\bar u|,|u|\leq C~{\rm for}~t\in [0,T],
    \label{priori41}
\end{eqnarray}
provided that $\delta_2$ and $\varepsilon_2$ are small enough.

Moreover, recalling that $v_\pm e^{\phi_\pm}=1$ and $s>0$, it holds that 
\begin{eqnarray}
k_1(\bar v)\ge\frac{R\theta_\pm+1}{v_\pm}>0,~ k_1(\bar
v)-e^{\phi_\pm}\ge\frac{R\theta_\pm}{v_\pm}>0.\label{prori42}
\end{eqnarray}

The $H^2$-estimates of $\sigma$ and the $H^1$-estimates of $\sigma_t$ can be obtained as in [11], yielding the following lemma.
\begin{lema}
Under the assumptions of Proposition 3.2, there is a positive
constant $C$ such that it holds that
\begin{eqnarray*}
&&\|\sigma\|^2_1\le C\|\varphi\|^2,\label{inequality41}\\
&&\|\sigma_{yy}\|^2\le C\|\varphi\|^2+C(\varepsilon_2+\delta_2)\|\varphi_y\|^2,\label{inequality42}\\
&&\|\sigma_t\|^2_1\le
C(\|\Phi_{yy}\|^2+\|\Psi_{yy}\|^2)+C(\varepsilon_2+\delta_2)\|\varphi\|^2\label{inequality43}
\end{eqnarray*}
for all $ t\in[0,T]$.
\end{lema}

\begin{rmk}
With Lemma 4.1 at hands, using (\ref{initial16}) and
(\ref{assumption12}), we immediately have  $\tilde{W}_{20}(x)\in
L^2(\mathbb{R})$. Furthermore, it holds that
$\tilde{W}_{0}=\tilde{W}_{10}+\tilde{W}_{20}\in L^2(\mathbb{R})$.
\end{rmk}

Now let us derive the basic energy estimate. First, multiplying
(\ref{model32})$_1$, (\ref{model32})$_2$ and (\ref{model32})$_3$ by
$\Phi, \frac{\bar{v}}{k_1(\bar v)}\Psi$ and $\frac{R}{k_1(\bar
v)(k_1(\bar v)-e^{\phi_\pm})}W$, respectively, 
and adding the resulting identities gives
\begin{eqnarray}
&&(\frac{1}{2}\Phi^2+\frac{\bar{v}}{2k_1(\bar
v)}\Psi^2+\frac{R^2}{2(\gamma-1)k_1(\bar v)(k_1(\bar
v)-e^{\phi_\pm})}W^2)_t
+B_{1y}+H_1\notag\\
&=&G_1+G_2+\frac{\bar{v}}{k_1(\bar v)}F_1\Psi+\frac{R}{k_1(\bar
v)(k_1(\bar v)-e^{\phi_\pm})}F_2W,\label{equality411}
\end{eqnarray}
where
\begin{eqnarray*}
B_1&=&-\frac{s}{2}\Phi^2-\frac{s\bar{v}}{2k_1(\bar
v)}\Psi^2-\frac{R^2s}{2(\gamma-1)k_1(\bar v)(k_1(\bar
v)-e^{\phi_\pm})}W^2-\Phi\Psi+\frac{R}{k_1(\bar v)}W\Psi
-\frac{\mu}{k_1(\bar v)}\Psi\Psi_y\\
&&-\frac{R\kappa}{\bar{v}k_1(\bar v)(k_1(\bar v)-e^{\phi_\pm})}(W_y+\frac{\gamma-1}{R}\bar{u}_y\Psi)W,\\
H_1&=&\big(\frac{\gamma-1}{k_1(\bar
v)}\bar{u}_y+(\frac{s\bar{v}}{2k_1(\bar v)})_y\big)\Psi^2
+\frac{\mu}{k_1(\bar v)}\Psi^2_y+(\frac{\mu}{k_1(\bar v)})_y\Psi\Psi_y\\
&&+\big(\frac{R^2s}{2(\gamma-1)k_1(\bar v)(k_1(\bar
v)-e^{\phi_\pm})}\big)_yW^2 +\frac{R\kappa}{\bar{v}k_1(\bar
v)(k_1(\bar v)-e^{\phi_\pm})}W^2_y-(\frac{R}{k_1(\bar v)})_y\Psi W,
\end{eqnarray*}
\begin{eqnarray*}
G_1&=&-\frac{\kappa(\gamma-1)}{\bar{v}k_1(\bar v)(k_1(\bar
v)-e^{\phi_\pm})}\bar{u}_y\Psi W_y -(\frac{R\kappa}{\bar{v}k_1(\bar
v)(k_1(\bar
v)-e^{\phi_\pm})})_y(W_y+\frac{\gamma-1}{R}\bar{u}_y\Psi)W\\
&&-\frac{R\kappa}{\bar{v}^2k_1(\bar v)(k_1(\bar v)-e^{\phi_\pm})}\bar{\theta}_y\Phi_yW,\\
G_2&=&\frac{k_2(\bar{v}, \bar{\phi}_y)}{k_1(\bar
v)}\Phi_y\Psi+\frac{Rk_2(\bar{v}, \bar{\phi}_y)}{k_1(\bar
v)(k_1(\bar v)-e^{\phi_\pm})}\Psi_yW.
\end{eqnarray*}
Then we have
\begin{lema}
Under the assumptions of Proposition 3.2, there is a positive
constant $C$ such that it holds for any $t\in[0,T]$,
\begin{eqnarray}
&&\|(\Phi,\Psi,W,\sigma,\sigma_y)(t)\|^2+\int_{0}^{t}(\|v^{\frac{1}{2}}_y(\Psi,
W,\sigma)(\tau)\|^2 +\|(\Psi_y,W_y)(\tau)\|^2)\mathrm{d}\tau\notag\\
&\le&C\|(\Phi_0,\Psi_0,W_0)\|^2+
C(\delta_2+\varepsilon_2)\int_{0}^{t}
\|(\Phi_y,\Phi_{yy},\Psi_{yy},w_y)(\tau)\|^2\mathrm{d}\tau.\label{estimate41}
\end{eqnarray}
\end{lema}
\noindent {\bf Proof.} It follows from the R-H condition
(\ref{relation12}) that
\begin{eqnarray*}
R(\theta_+-\theta_-)=\frac{R\theta_-+1}{v_-}(v_+-v_-)-s^2v_+(v_+-v_-),
\end{eqnarray*}
and
\begin{eqnarray*}
-\frac{(R\theta_-+1)\gamma}{(\gamma-1)v_-}(v_+-v_-)+\frac{s^2v_+}{\gamma-1}(v_+-v_-)
-(\phi_+-\phi_-)+\frac{s^2}{2}(v_+-v_-)^2=0.
\end{eqnarray*}
Namely,
\begin{eqnarray}
s^2=\frac{(R\theta_-+1)\gamma}{v_+v_-}(1-\frac{d_+}{1+d_+})
-\frac{(\ln{v_+}-\ln{v_-})(\gamma-1)}{(v_+-v_-)v_+}\frac{1}{1+d_+},\label{relation41}
\end{eqnarray}
and
\begin{eqnarray}
\theta_+=\theta_-(1-\frac{v_++v_-}{v_-}\frac{d_+}{1+d_+})
-\frac{v_++v_-}{Rv_-}\frac{d_+}{1+d+}+\frac{\gamma-1}{1+d_+}(\ln{v_+}-\ln{v_-})\label{relation42}
\end{eqnarray}
with $d_+=\frac{(\gamma-1)(v_+-v_-)}{2v_+}$. Then from
(\ref{relation41}) and (\ref{relation42}), similar to the arguments
in \cite{KM}, we can show
\begin{eqnarray}
H_1\ge
c(\bar{v}_y\Psi^2+\bar{v}_yW^2+\Psi^2_y+W^2_y),\label{estimate411}
\end{eqnarray}
provided that $\delta_2$ is small enough. Then integrating
(\ref{equality411}) with respect to $y$ over $\mathbb{R}$ and using
(\ref{priori41}), (\ref{prori42}) and (\ref{estimate411}) yield
\begin{eqnarray}
&&\frac{\mathrm{d}}{\mathrm{d}t}\int_{\mathbb{R}}(\Phi^2+\Psi^2+W^2)\mathrm{d}y+\int_{\mathbb{R}}(\bar
v_y\Psi^2+\bar v_yW^2+\Psi_y^2+W_y^2)\mathrm{d}y  -2\int_{\mathbb{R}}\frac{\bar{v}}{k_1(\bar{v})}F_1\Psi\mathrm{d}y \notag\\
&\leq&C\Big|\int_{\mathbb{R}}(G_1+G_2 
+\frac{R}{k_1(\bar{v})(k_1(\bar{v})-e^{\phi_\pm})}F_{2}W)\mathrm{d}y  \Big|.\label{equality41}
\end{eqnarray}

Now let us control each term on the right-hand side of
(\ref{equality41}). First, utilizing (\ref{priori41}),
(\ref{prori42}), (\ref{inequality12}) and the Young inequality, one
gets
\begin{eqnarray}
|\int_{\mathbb{R}}G_1\mathrm{d}y|\le\frac{1}{8}\int_{\mathbb{R}}(\bar{v}_yW^2+W^2_y)\mathrm{d}y
+C\delta_2\int_{\mathbb{R}}(\bar{v}_y\Psi^2+\bar{v}_yW^2+\Phi^2_y)\mathrm{d}y.\label{estimate412}
\end{eqnarray}
Moreover, from the definition of $k_1(\bar{v})$ and $k_2(\bar{v},
\bar{\phi}_y)$, we note that
\begin{eqnarray*}
G_2=-\big(\frac{\lambda^2}{\bar{v}^2k_1(\bar{v})}
\bar{\phi}_y\Phi_y\Psi+\frac{\lambda^2}{\bar{v}^2k_1(\bar{v})(k_1(\bar
v)-e^{\phi_\pm})} \bar{\phi}_y\Psi_yW\big)_y+\tilde{G}_2
\end{eqnarray*}
with
\begin{eqnarray*}
\tilde{G}_2&=&\frac{\lambda^2}{2k_1(\bar{v})}(\frac{\bar{\phi}_y}{\bar{v}})^2\Phi_y\Psi
+(\frac{\lambda^2}{\bar{v}k_1(\bar{v})})_y\frac{\bar{\phi}_y}{\bar{v}}\Phi_y\Psi
+\frac{\lambda^2}{2k_1(\bar{v})(k_1(\bar
v)-e^{\phi_\pm})}(\frac{\bar{\phi}_y}{\bar{v}})^2\Psi_yW\notag\\
&& +(\frac{\lambda^2}{\bar{v}k_1(\bar{v})(k_1(\bar
v)-e^{\phi_\pm})})_y\frac{\bar{\phi}_y}{\bar{v}}\Psi_yW
+\frac{\lambda^2}{\bar{v}^2k_1(\bar{v})}\bar{\phi}_y(\Phi_{yy}\Psi+\Phi_y\Psi_y)\notag\\
&& +\frac{\lambda^2}{\bar{v}^2k_1(\bar{v})(k_1(\bar
v)-e^{\phi_\pm})}\bar{\phi}_y(\Psi_{yy}W+\Psi_yW_y).
\end{eqnarray*}
Then making use of (\ref{inequality11}), (\ref{inequality12}) and
the Cauchy inequality, we can obtain
\begin{eqnarray}
|\int_{\mathbb{R}}G_2\mathrm{d}y|&=&|\int_{\mathbb{R}}\tilde{G}_2\mathrm{d}y|
\leq C|\int_{\mathbb{R}}(\bar\phi_y^2+\bar
v_y\bar\phi_y)(\Phi_y\Psi+\Psi_yW)\mathrm{d}y|\notag\\
&&+C|\int_{\mathbb{R}}\bar\phi_y(\Phi_{yy}\Psi+\Psi_{yy}W+\Phi_y\Psi_y+\Psi_yW_y)\mathrm{d}y| \notag\\
&\le&C\delta_2\int_{\mathbb{R}}(\bar{v}_y\Psi^2+\bar{v}_yW^2)\mathrm{d}y+C\delta_2(\|\Phi_y\|_1^2
+\|\Psi_y\|_1^2+\|W_y\|^2).\label{estimate413}
\end{eqnarray}

Now let us focus on $\frac{\bar{v}}{k_1(\bar{v})}F_1\Psi$. First,
from (\ref{priori41}), (\ref{prori42}), (\ref{inequality12}),
(\ref{assumption31}) and the Sobolev inequality and the Young
inequality, one gets
\begin{eqnarray}
|\int_{\mathbb{R}}\frac{\bar{v}}{k_1(\bar{v})}F_{11}\Psi
\mathrm{d}y|&\leq&
C\|\Psi\|_{L^\infty}\int_{\mathbb{R}}(\Phi^2_y+\Psi_y^2+\Psi_{yy}^2+w^2+\sigma_y^2)\mathrm{d}y
+C|\int_{\mathbb{R}}\bar\phi_y(\Psi\sigma_y+\Phi_y\Psi)\mathrm{d}y|\notag\\
&&+C\|\Psi\|_{L^\infty}(\|\bar\phi_y\|_{L^\infty}^2
+\|\bar\phi_{yy}\|_{L^\infty}+\|\bar v_y\|_{L^\infty}\|\bar\phi_y\|_{L^\infty})\int_{\mathbb{R}}\Phi^2_y\mathrm{d}y\notag\\
&\leq&\frac{1}{8}\int_{\mathbb{R}}\bar{v}_y\Psi^2\mathrm{d}y+
C(\delta_2+\varepsilon_2)(\|\Phi_y\|^2+\|\Psi_y\|_1^2+\|w\|^2+\|\sigma_y\|^2).\label{estimate4141}
\end{eqnarray}
Next, using (\ref{model13})$_4$, it is easy to see
\begin{eqnarray*}
\frac{\bar{v}}{k_1(\bar{v})}F_{12}\Psi&=&\frac{\gamma-1}{k_1(\bar{v})}\big(\sigma(v_\pm
e^{\phi_\pm}-\bar{v}e^{\bar{\phi}})+\bar{v}e^{\bar{\phi}}(1+\sigma-e^\sigma)+\Phi_y(e^{\phi_\pm}-e^{\bar{\phi}})
+e^{\bar{\phi}}\Phi_y(e^{\phi_\pm}-e^\sigma) \notag\\
&&-\frac{\lambda^2}{2}\big(\frac{\sigma^2_y+2\bar{\phi}_y\sigma_y}{\bar{v}+\Phi_y}-
\frac{\bar{\phi}^2_y\Phi_y}{\bar{v}(\bar{v}+\Phi_y)}\big)\big)\Psi \notag\\
&=&-\big(\frac{\lambda^2(\gamma-1)}{k_1(\bar{v})\bar
v}\bar\phi_y\Psi\sigma\big)_{y}+G_3,\label{equality412}
\end{eqnarray*}
here
\begin{eqnarray*}
G_3&=&(\frac{\lambda^2(\gamma-1)}{k_1(\bar{v})})_y\frac{\bar\phi_y}{\bar
v}\Psi\sigma+\frac{\lambda^2(\gamma-1)}{k_1(\bar{v})\bar
v}\bar\phi_y(\Psi_y\sigma+\Psi\sigma_y)+\frac{\gamma-1}{k_1(\bar{v})}\big(\bar{v}e^{\bar{\phi}}(1+\sigma-e^\sigma) \notag\\
&&+\Phi_y(e^{\phi_\pm}-e^{\bar{\phi}})
+e^{\bar{\phi}}\Phi_y(1-e^\sigma)-\frac{\lambda^2}{2}(\frac{\sigma^2_y+2\bar{\phi}_y\sigma_y}{\bar{v}+\Phi_y}-
\frac{\bar{\phi}^2_y\Phi_y}{\bar{v}(\bar{v}+\Phi_y)})\big)\Psi.
\end{eqnarray*}
Then same as (\ref{estimate4141}), one gets
\begin{align}
|\int_{\mathbb{R}}\frac{\bar{v}}{k_1(\bar{v})}F_{12}\Psi
\mathrm{d}y|&=|\int_{\mathbb{R}}G_3\mathrm{d}y| \leq
C|\int_{\mathbb{R}}(\bar
v_y\bar\phi_y\Psi\sigma+\bar\phi_y\Psi_y\sigma+\bar\phi_y\Psi\sigma_y+\sigma^2\Psi+\bar\phi_yy\Phi_y\Psi
\notag\\
&\qquad\qquad\qquad\quad +\Phi_y\Psi\sigma+\sigma_y^2\Psi+\bar\phi_y\sigma_y\Psi+\bar\phi_y^2\Phi_y\Psi)\mathrm{d}y| \notag\\
&\le\frac{1}{8}\int_{\mathbb{R}}\bar{v}_y\Psi^2\mathrm{d}y
+C(\delta_2+\varepsilon_2)(\|\sigma\|_1^2+\|\Phi_y\|^2+\|\Psi_y\|^2).\label{estimate4142}
\end{align}
Finally, let us carefully deal with the term
$-\frac{\bar{v}}{k_1(\bar{v})}F_{13}\Psi$, which provides the energy for $\sigma$. First, with the help of
(\ref{model32})$_1$, we can rewrite it as
\begin{eqnarray}
\frac{\bar{v}}{k_1(\bar{v})}F_{13}\Psi&=&-\frac{\lambda^2\bar{v}}{k_1(\bar{v})}\big(\frac{1}{\bar{v}
+\Phi_y}(\frac{\bar{\phi}_y+\sigma_y}{\bar{v}+\Phi_y})_y
-\frac{1}{\bar{v}}(\frac{\bar{\phi}_y}{\bar{v}})_y\big)\Psi \notag\\
&=&-\frac{\lambda^2\bar{v}}{k_1(\bar{v})}\big(\big(\frac{\bar{\phi}_y+\sigma_y}{(\bar{v}+\Phi_y)^2}-\frac{\bar{\phi}_y}{\bar{v}^2}\big)_y
-\big(\frac{1}{\bar{v}+\Phi_y}\big)_y\frac{\bar{\phi}_y+\sigma_y}{\bar{v}+\Phi_y}+\big(\frac{1}{\bar{v}}\Big)_y\frac{\bar{\phi}_y}{\bar{v}}\big)\Psi \notag\\
&=&-\big(\frac{\lambda^2\bar{v}}{k_1(\bar{v})}\Psi\big(\frac{\bar{\phi}_y+\sigma_y}{(\bar{v}+\Phi_y)^2}-\frac{\bar{\phi}_y}{\bar{v}^2}\big)\big)_y
+\frac{\lambda^2\bar{v}}{k_1(\bar{v})}\big(\frac{\bar{\phi}_y+\sigma_y}{(\bar{v}+\Phi_y)^2}-\frac{\bar{\phi}_y}{\bar{v}^2}\big)\Psi_y \notag\\
&&+(\frac{\lambda^2\bar{v}}{k_1(\bar{v})})_y\Psi\big(\frac{\bar{\phi}_y+\sigma_y}{(\bar{v}+\Phi_y)^2}-\frac{\bar{\phi}_y}{\bar{v}^2}\big)
+\frac{\lambda^2\bar{v}}{k_1(\bar{v})}\Psi\big((\frac{1}{\bar{v}+\Phi_y})_y\frac{\bar{\phi}_y+\sigma_y}{\bar{v}+\Phi_y}
-(\frac{1}{\bar{v}})_y\frac{\bar{\phi}_y}{\bar{v}}\big) \notag\\
&=&-\lambda^2\big(\frac{\bar{v}}{k_1(\bar{v})}\Psi(\frac{\bar{\phi}_y+\sigma_y}{(\bar{v}+\Phi_y)^2}-\frac{\bar{\phi}_y}{\bar{v}^2})\big)_y
+\lambda^2\sigma_y\Psi_y+\lambda^2(\frac{1}{k_1(\bar{v})\bar{v}}-1)\sigma_y\Psi_y\notag\\
&&+\frac{\lambda^2\bar{v}}{k_1(\bar{v})}\sigma_y\Psi_y\big(\frac{1}{(\bar{v}+\Phi_y)^2}-\frac{1}{\bar{v}^2}\big)
+\frac{\lambda^2\bar{v}}{k_1(\bar{v})}\bar{\phi}_y\big(\frac{1}{(\bar{v}+\Phi_y)^2}-\frac{1}{\bar{v}^2}\big)\Psi_y\notag\\
&&
+(\frac{\lambda^2\bar{v}}{k_1(\bar{v})})_y\Psi\big(\frac{\sigma_y}{(\bar{v}+\Phi_y)^2}
+\bar{\phi}_y\big(\frac{1}{(\bar{v}+\Phi_y)^2}-\frac{1}{\bar{v}^2}\big)\big) \notag\\
&&-\frac{\lambda^2\bar{v}}{k_1(\bar{v})}\big(\frac{\bar{\phi}_y\Phi_{yy}+\bar{v}_y\sigma_y+\sigma_y\Phi_{yy}}{(\bar{v}+\Phi_y)^3}
+\bar{v}_y\bar{\phi}_y\big(\frac{1}{(\bar{v}+\Phi_y)^3}
-\frac{1}{\bar{v}^3}\big)\big)\Psi\notag\\
&=&B_{21y}+\lambda^2(\sigma_y\Phi)_t +\lambda^2\sigma_t\Phi_y
-\lambda^2s\sigma_y\Phi_y+G_{41},\label{equality413}
\end{eqnarray}
here
\begin{eqnarray*}
B_{21}&=&-\frac{\lambda^2\bar{v}}{k_1(\bar{v})}\Psi\big(\frac{\bar{\phi}_y+\sigma_y}{(\bar{v}+\Phi_y)^2}-\frac{\bar{\phi}_y}{\bar{v}^2}\big)
- \lambda^2\sigma_t\Phi,\\
G_{41}&=&\lambda^2\big(\frac{1}{k_1(\bar{v})\bar{v}}-1\big)\sigma_y\Psi_y
+\frac{\lambda^2\bar{v}}{k_1(\bar{v})}\sigma_y\Psi_y\big(\frac{1}{(\bar{v}+\Phi_y)^2}-\frac{1}{\bar{v}^2}\big)
+\frac{\lambda^2\bar{v}}{k_1(\bar{v})}\bar{\phi}_y\big(\frac{1}{(\bar{v}+\Phi_y)^2}-\frac{1}{\bar{v}^2}\big)\Psi_y \notag\\
&&+(\frac{\lambda^2\bar{v}}{k_1(\bar{v})})_y\Psi\big(\frac{\sigma_y}{(\bar{v}+\Phi_y)^2}
+\bar{\phi}_y\big(\frac{1}{(\bar{v}+\Phi_y)^2}-\frac{1}{\bar{v}^2}\big)\big) \notag\\
&&-\frac{\lambda^2\bar{v}}{k_1(\bar{v})}\big(\frac{\bar{\phi}_y\Phi_{yy}+\bar{v}_y\sigma_y+\sigma_y\Phi_{yy}}{(\bar{v}+\Phi_y)^3}
+\bar{v}_y\bar{\phi}_y\big(\frac{1}{(\bar{v}+\Phi_y)^3}
-\frac{1}{\bar{v}^3}\big)\big)\Psi\notag\\
&\sim&O(|\bar
v-v_\pm|\sigma_y\Psi_y+\sigma_y\Phi_y\Psi_y+\bar\phi_y\Phi_y\Psi_y+\bar
v_y\Psi\sigma_y+\bar
v_y\bar\phi_y\Psi\Phi_y+\bar\phi_y\Psi\Phi_{yy}+\bar
v_y\Psi\sigma_y+\Psi\sigma_y\Phi_{yy}).
\end{eqnarray*}

Moreover, using (\ref{model32})$_4$, we find that
\begin{eqnarray}
\lambda^2\sigma_t\Phi_y&=&\lambda^4e^{-\bar{\phi}}
\big(\frac{\bar{\phi}_y+\sigma_y}{\bar{v}+\Phi_y}-\frac{\bar{\phi}_y}{\bar{v}}\big)_y\sigma_t
-\lambda^2\bar v\sigma\sigma_t+\lambda^2e^{-\bar{\phi}}F_3\sigma_t \notag\\
&=&\big(\lambda^4e^{-\bar{\phi}}\big(\frac{\bar{\phi}_y+\sigma_y}{\bar{v}+\Phi_y}-\frac{\bar{\phi}_y}{\bar{v}}\big)
\sigma_t\big)_y-(\lambda^4e^{-\bar{\phi}})_y\big(\frac{\bar{\phi}_y+\sigma_y}{\bar{v}+\Phi_y}-\frac{\bar{\phi}_y}{\bar{v}}\big)
\sigma_t \notag\\
&&-\lambda^4e^{-\bar{\phi}}\big(\frac{\bar{\phi}_y+\sigma_y}{\bar{v}+\Phi_y}-\frac{\bar{\phi}_y}{\bar{v}}\big)
\sigma_{ty}-\lambda^2\bar v\sigma\sigma_t+\lambda^2e^{-\bar{\phi}}F_3\sigma_t \notag\\
&=&-\big(\frac{\lambda^2\bar
v}{2}\sigma^2+\frac{\lambda^4}{2\bar{v}}e^{-\bar{\phi}}\sigma^2_y\big)_t+B_{22y}+\lambda^2e^{-\bar{\phi}}F_3\sigma_t+G_{42},\label{equality414}
\end{eqnarray}
and
\begin{eqnarray}
-\lambda^2s\sigma_y\Phi_y&=&-\lambda^2s\sigma_y\big(\lambda^2e^{-\bar{\phi}}
\big(\frac{\bar{\phi}_y+\sigma_y}{\bar{v}+\Phi_y}
-\frac{\bar{\phi}_y}{\bar{v}}\big)_y-\bar{v}\sigma+e^{-\bar{\phi}}F_3\big)\notag\\
&=&-\big(\lambda^4s\sigma_ye^{-\bar{\phi}}\big(\frac{\bar{\phi}_y+\sigma_y}{\bar{v}+\Phi_y}
-\frac{\bar{\phi}_y}{\bar{v}}\big)\big)_y
+(\lambda^4se^{-\bar{\phi}})_y\big(\frac{\bar{\phi}_y+\sigma_y}{\bar{v}+\Phi_y}
-\frac{\bar{\phi}_y}{\bar{v}}\big)\sigma_y\notag\\
&&+\lambda^2s\bar
v\sigma\sigma_y-\lambda^2se^{-\bar{\phi}}\sigma_yF_3
+\lambda^4se^{-\bar{\phi}}\big(\frac{\bar{\phi}_y+\sigma_y}{\bar{v}+\Phi_y}
-\frac{\bar{\phi}_y}{\bar{v}}\big)\sigma_{yy} \notag\\
&=&-\frac{\lambda^2s}{2}\bar
v_y\sigma^2+B_{23y}-\lambda^2se^{-\bar{\phi}}\sigma_yF_3+G_{43},\label{equality415}
\end{eqnarray}
here
\begin{eqnarray*}
B_{22}&=&\lambda^4e^{-\bar{\phi}}\big(\frac{\bar{\phi}_y+\sigma_y}{\bar{v}+\Phi_y}-\frac{\bar{\phi}_y}{\bar{v}}\big)
\sigma_t,\\
G_{42}&=&-(\lambda^4e^{-\bar{\phi}})_y\big(\frac{\bar{\phi}_y+\sigma_y}{\bar{v}+\Phi_y}-\frac{\bar{\phi}_y}{\bar{v}}\big)
\sigma_t
-\lambda^4e^{-\bar{\phi}}\sigma_y\sigma_{ty}\big(\frac{1}{\bar{v}+\Phi_y}-\frac{1}{\bar{v}}\big)
-\lambda^4e^{-\bar{\phi}}\bar{\phi}_y\sigma_{ty}\big(\frac{1}{\bar{v}+\Phi_y}-\frac{1}{\bar{v}}\big)\\
&\sim&O(\bar\phi_y\sigma_y\sigma_t+\bar\phi_y^2\Phi_y\sigma_t+\sigma_y\Phi_y\sigma_{ty}+\bar\phi_y\Phi_y\sigma_{ty}),\\
B_{23}&=&
-\lambda^4s e^{-\bar{\phi}}\big(\frac{\bar{\phi}_y+\sigma_y}{\bar{v}+\Phi_y}-\frac{\bar{\phi}_y}{\bar{v}}\big)\sigma_y
+\frac{\lambda^2s\bar v}{2}\sigma^2+\frac{\lambda^4s}{2\bar{v}}e^{-\bar{\phi}}\sigma^2_y, \notag\\
G_{43}&=&(\lambda^4se^{-\bar{\phi}})_y\big(\frac{\bar{\phi}_y+\sigma_y}{\bar{v}+\Phi_y}-\frac{\bar{\phi}_y}{\bar{v}}\big)\sigma_y
+\lambda^4se^{-\bar{\phi}}\bar{\phi}_y\sigma_{yy}\big(\frac{1}{\bar{v}+\Phi_y}-\frac{1}{\bar{v}}\big)
+\lambda^4se^{-\bar{\phi}}\sigma_y\big(\frac{1}{\bar{v}+\Phi_y}-\frac{1}{\bar{v}}\big)\sigma_{yy} \notag\\
&&-(\frac{\lambda^4s}{2\bar{v}}e^{-\bar{\phi}})_y\sigma^2_y\sim
O(\bar\phi_y\sigma_y^2+\bar\phi_y^2\Phi_y\sigma_y+\sigma_y\Phi_y\sigma_{yy}+\bar\phi_y\Phi_y\sigma_{yy}).
\end{eqnarray*}
Further, same as (\ref{estimate413}) and (\ref{estimate4141}), one
gets
\begin{eqnarray}
&&\Big|\int_{\mathbb{R}}G_{41}\mathrm{d}y\Big|\le\frac{1}{8}\int_{\mathbb{R}}\bar{v}_y\Psi^2\mathrm{d}y
+C(\delta_2+\varepsilon_2)(\|\sigma_y\|^2+\|\Phi_y\|_1^2+\|\Psi_y\|^2),\label{estimate4143-1}\\
&&\Big|\int_{\mathbb{R}}(G_{42}+G_{43})\mathrm{d}y\Big|
\le
C(\delta_2+\varepsilon_2)(\|\sigma_y\|^2+\|\sigma_t\|^2+\|\Phi_y\|^2+\|\sigma_{ty}\|^2),\label{estimate4143-2}
\end{eqnarray}
and
\begin{eqnarray}
\Big|\lambda^2\int_{\mathbb{R}}e^{-\bar{\phi}}(F_3\sigma_t+s\sigma_yF_3)\mathrm{d}y\Big|\le
C\varepsilon_2(\|\Phi_y\|^2+\|\sigma\|_1^2+\|\sigma_t^2).\label{estimate4143-3}
\end{eqnarray}
Hence, we obtain
\begin{eqnarray}
-\int_{\mathbb{R}}\frac{\bar{v}}{k_1(\bar{v})}F_{13}\Psi
\mathrm{d}y
&\leq& \frac{\mathrm{d}}{\mathrm{d}t}\int_{\mathbb{R}}\big(\frac{\lambda^2\bar
v}{2}\sigma^2+\frac{\lambda^4}{2\bar{v}}e^{-\bar{\phi}}\sigma^2_y\big)\mathrm{d}y + \frac{\lambda^2s}{2}\int_{\mathbb{R}}\bar
v_y\sigma^2\mathrm{d}y+\frac{1}{8}\int_{\mathbb{R}}\bar{v}_y\Psi^2\mathrm{d}y\notag\\
&&+C(\delta_2+\varepsilon_2)(\|\sigma\|_2^2+\|\Phi_y\|_1^2
+\|\Psi_y\|^2+\|\sigma_t\|_1^2).\label{estimate4143}
\end{eqnarray}

Finally, let us deal with the terms
$\frac{R}{k_1(\bar{v})(k_1(\bar{v})-e^{\phi_\pm})}F_{21}W$ and
$\frac{R}{k_1(\bar{v})(k_1(\bar{v})-e^{\phi_\pm})}F_{22}W$. First,
using (\ref{priori41}) and (\ref{prori42}), it is easy to see
\begin{eqnarray*}
F_{21}&\sim&
O(1)(\Psi_y\Psi_{yy}+\Phi_y\Psi_y+(\bar\phi_{yy}+\bar\phi_{y}^2+\bar\phi_{y}\bar\phi_{y})\Phi_y\Psi_y+\Psi_yw+\Phi_yw_y+\bar
\theta_y\Phi_y^2+(\bar\phi_y+\sigma_y)\sigma_t\\
&&+\bar \phi_y\sigma_y+\sigma_y^2+\bar \phi_y^2(\Phi_y+\Psi_y)+(\bar
\phi_y+\sigma_y)\Psi_y\sigma_y+(\bar\phi_{yy}+\sigma_{yy})\Psi_y+(\bar\phi_{y}+\sigma_{y})\bar
v_{y}\Psi_y\\
&&+(\bar\phi_{y}+\sigma_{y})\Phi_{yy}\Psi_y),
\end{eqnarray*}
then same as (\ref{estimate413}) and (\ref{estimate4141}), one gets
\begin{eqnarray}
\Big|\int_{\mathbb{R}}\frac{R}{k_1(\bar{v})(k_1(\bar{v})-e^{\phi_\pm})}F_{21}W\mathrm{d}y \Big|
&\le&\frac{1}{8}\int_{\mathbb{R}}\bar{v}_yW^2\mathrm{d}y+
C(\delta_2+\varepsilon_2)(\|\Phi_y\|_1^2+\|\Psi_y\|_1^2+\|w\|^2+\|w_y\|^2\notag\\
&&+\|\sigma_y\|_1^2
+\|W_{y}\|^2+\|\sigma_t\|^2).\label{estimate4144}
\end{eqnarray}
Next,  we can rewrite the term
$\frac{R}{k_1(\bar{v})(k_1(\bar{v})-e^{\phi_\pm})}F_{22}W$ as
\begin{eqnarray*}
&&\frac{R}{k_1(\bar{v})(k_1(\bar{v})-e^{\phi_\pm})}F_{22}W\notag\\
&=&-\Big(\frac{\kappa(\gamma-1)}{\bar
vk_1(\bar{v})(k_1(\bar{v})-e^{\phi_\pm})}
\big(\sigma-(\bar{v}+\Phi_y)e^{\bar{\phi}+\sigma}+\bar{v}e^{\bar{\phi}}+e^{\phi_\pm}\Phi_y
-\frac{\lambda^2}{2}(\frac{(\bar{\phi}_y+\sigma_y)^2}{\bar{v}+\Phi_y}-\frac{\bar{\phi}^2_y}{\bar{v}}\big)\Big)W\big)_y+G_5,
\end{eqnarray*}
here
\begin{eqnarray*}
G_5&=&\big(\frac{\kappa(\gamma-1)}{\bar
vk_1(\bar{v})(k_1(\bar{v})-e^{\phi_\pm})}\big)_y
\big(\sigma-(\bar{v}+\Phi_y)e^{\bar{\phi}+\sigma}+\bar{v}e^{\bar{\phi}}+e^{\phi_\pm}\Phi_y
-\frac{\lambda^2}{2}(\frac{(\bar{\phi}_y+\sigma_y)^2}{\bar{v}+\Phi_y}-\frac{\bar{\phi}^2_y}{\bar{v}})\big)W\notag\\
&&+\frac{\kappa(\gamma-1)}{\bar
vk_1(\bar{v})(k_1(\bar{v})-e^{\phi_\pm})}
\big(\sigma-(\bar{v}+\Phi_y)e^{\bar{\phi}+\sigma}+\bar{v}e^{\bar{\phi}}+e^{\phi_\pm}\Phi_y
-\frac{\lambda^2}{2}(\frac{(\bar{\phi}_y+\sigma_y)^2}{\bar{v}+\Phi_y}-\frac{\bar{\phi}^2_y}{\bar{v}})\big)W_y\\
&\sim& O(1)(\bar v_y\bar\phi_{y}\sigma
W+\bar\phi_yW_y\sigma+\bar\phi_yW\sigma_y+\sigma^2W+\bar\phi_yy\Phi_yW+\Phi_y\sigma
W+\sigma^2_yW+\bar\phi_y\Phi_yW+\sigma^2W_y\\
&&+\bar\phi_y^2\Phi_yW+\bar\phi_{yy}\sigma W_y+\bar
v_y\bar\phi_{y}\sigma W_y+\bar\phi_{y}y\Phi_yW_y+\Phi_y\sigma
W_y+\sigma_y^2W_y+\bar\phi_{y}\sigma_yW_y+\bar\phi_{y}\Phi_yW_y) .
\end{eqnarray*}
Then similar to (\ref{estimate4142}), we can obtain
\begin{eqnarray*}
\Big|\int_{\mathbb{R}}G_{5}\mathrm{d}y \Big|\le\frac{1}{8}\int_{\mathbb{R}}(\bar{v}_yW^2+W_y^2)\mathrm{d}y
+C(\delta_2+\varepsilon_2)(\|\sigma\|_1^2+\|\Phi_y\|^2+\|W_{y}\|^2).\label{estimate4152}
\end{eqnarray*}
That is,
\begin{eqnarray}
\Big|\int_{\mathbb{R}}\frac{R}{k_1(\bar{v})(k_1(\bar{v})-e^{\phi_\pm})}F_{22}W\mathrm{d}y\Big|
&\le&\frac{1}{8}\int_{\mathbb{R}}(\bar{v}_yW^2+W_y^2)\mathrm{d}y\notag\\
&&+C(\delta_2+\varepsilon_2)(\|\sigma\|_1^2+\|\Phi_y\|^2+\|W_{y}\|^2).\label{estimate4145}
\end{eqnarray}
Moreover, recalling (\ref{relation36}), it is trivial to get
\begin{eqnarray}
\|w\|^2\leq C(\|\bar
v_y^{\frac12}\Psi\|^2+\|\sigma\|^2+\|\sigma_y\|^2+\|\Phi_y\|^2+\|W_y\|^2).\label{estimate4146}
\end{eqnarray}

Therefore, putting (\ref{estimate412}), (\ref{estimate413}),
(\ref{estimate4141}), (\ref{estimate4142}),
(\ref{estimate4143})-(\ref{estimate4145}) and (\ref{estimate4146})
into (\ref{equality41}), then integrating the resultant inequality
with respect to $t$ and using Lemma 4.1, we immediately have
(\ref{estimate41}). This completes the proof.

\

Next we estimate the $L^2([0,T];L^2)$-norm of $\Phi_y(=\varphi)$.
\begin{lema}
Under the assumptions of Proposition 3.2, there is a positive
constant $C$ such that it holds for any $t\in[0,T]$,
\begin{eqnarray}
\|\Phi_y(t)\|^2+\int_{0}^{t}\|(\Phi_y,\sigma_y,\sigma_{yy})(\tau)\|^2\mathrm{d}\tau&\le&
C(\|\Phi_0\|_1^2+\|\Psi_0\|^2+\|W_0\|^2)\notag\\
&&+C(\delta_2+\varepsilon_2)\int_{0}^{t}
\|(\varphi_y,\psi_y,w_y)(\tau)\|^2\mathrm{d}\tau.\label{estimate42}
\end{eqnarray}
\end{lema}
\noindent {\bf Proof.} 
Multiplying the equation $\Phi$ and $\Psi$ in
(\ref{model32}) by $\bar{v}\Psi_y-\bar{v}_y\Psi$ and
$-\bar{v}\Phi_y$, respectively. 
Differentiating $(3.13)_1$ with respect to $y$ and multiplying the
resultant by $\mu\Phi_y$, then adding the three identities, gives 
\begin{eqnarray}
&&(\frac{\mu}{2}\Phi^2_y-\bar{v}\Phi_y\Psi)_t+(k_1+k_2)\Phi_y^2+(\bar{v}\Psi\Phi_t-\frac{\mu
s}{2}\Phi^2_y)_y - \bar{v}_y(s\Phi_y+\Psi_y)\Psi-\bar{v}\Psi^2_y \notag\\
&=&\Phi_y(RW_y+(\gamma-1)\bar{u}_y\Psi)-\bar{v}\Phi_y(F_{11}+F_{12})+\frac{\lambda^2}{\bar{v}}\Phi_y\sigma_{yy}+G_6,\label{equality421}
\end{eqnarray}
here we used the fact
\begin{eqnarray*}
-\bar{v}\Phi_yF_{13}&=&\lambda^2\bar{v}\Phi_y\big(\frac{1}{\bar{v}+\Phi_y}(\frac{\bar{\phi}_y+\sigma_y}{\bar{v}+\Phi_y})_y
-\frac{1}{\bar{v}}(\frac{\bar{\phi}_y}{\bar{v}})_y\big) \notag\\
&=&\lambda^2\bar{v}\Phi_y\big(\frac{1}{\bar{v}+\Phi_y}(\frac{\bar{\phi}_y+\sigma_y}{\bar{v}+\Phi_y}-\frac{\bar{\phi}_y}{\bar{v}})_y
+(\frac{\bar{\phi}_y}{\bar{v}})_y(\frac{1}{\bar{v}+\Phi_y}-\frac{1}{\bar{v}})\big) \notag\\
&=&\frac{\lambda^2\bar{v}\Phi_y}{(\bar{v}+\Phi_y)^2}\sigma_{yy}+\frac{\lambda^2\bar{v}\Phi_y}{\bar{v}+\Phi_y}\bar{\phi}_{yy}
\big(\frac{1}{\bar{v}+\Phi_y}-\frac{1}{\bar{v}}\big)+\lambda^2\bar{v}\Phi_y(\frac{\bar{\phi}_y}{\bar{v}})_y
\big(\frac{1}{\bar{v}+\Phi_y}-\frac{1}{\bar{v}}\big) \notag\\
&&-\frac{\lambda^2\bar{v}\Phi_y}{\bar{v}+\Phi_y}\big(\frac{\bar{\phi}_y\Phi_{yy}+\bar{v}_y\sigma_y+\Phi_{yy}\sigma_y}{(\bar{v}+\Phi_y)^2}
+\bar{\phi}_y\bar{v}_y(\frac{1}{(\bar{v}+\Phi_y)^2}-\frac{1}{\bar{v}^2})\big) \notag\\
&=&\frac{\lambda^2}{\bar{v}}\sigma_{yy}\Phi_{y}+G_6
\end{eqnarray*}
with
\begin{eqnarray*}
G_6&=&\lambda^2\bar{v}\Phi_y\sigma_{yy}\big(\frac{1}{(\bar{v}+\Phi_y)^2}-\frac{1}{\bar{v}^2}\big)
+\frac{\lambda^2\bar{v}\Phi_y}{\bar{v}+\Phi_y}
\bar{\phi}_{yy}\big(\frac{1}{\bar{v}+\Phi_y}-\frac{1}{\bar{v}}\big)+\lambda^2\bar{v}\Phi_y(\frac{\bar{\phi}_y}{\bar{v}})_y
\big(\frac{1}{\bar{v}+\Phi_y}-\frac{1}{\bar{v}}\big) \notag\\
&&-\frac{\lambda^2\bar{v}\Phi_y}{(\bar{v}+\Phi_y)^3}(\bar{\phi}_y\Phi_{yy}+\bar{v}_y\sigma_y+\Phi_{yy}\sigma_y)
-\frac{\lambda^2\bar{v}}{\bar{v}+\Phi_y}\bar{\phi}_y\bar{v}_y\Phi_y\big(\frac{1}{(\bar{v}+\Phi_y)^2}-\frac{1}{\bar{v}^2}\big)\\
&\sim&O(\Phi_y^2\sigma_{yy}+\bar\phi_{yy}\Phi_y^2+\bar
v_y\bar\phi_y\Phi_y^2+\bar\phi_y\Phi_y\Phi_{yy}+\bar
v_y\Phi_y\sigma_y+ \bar
v_y\bar\phi_y\Phi_y^2+\sigma_y\Phi_y\Phi_{yy}).
\end{eqnarray*}
Moreover, multiplying (\ref{model32})$_4$ by
$-\frac{\lambda^2}{\bar{v}}e^{-\bar\phi}\sigma_{yy}$, one has
\begin{eqnarray}
\frac{\lambda^4}{\bar{v}}\sigma^2_{yy}+\lambda^2\sigma^2_y-\frac{\lambda^2}{\bar{v}}\Phi_y\sigma_{yy}-(\lambda^2\sigma\sigma_y)_y
=G_7,\label{equality422}
\end{eqnarray}
here
\begin{eqnarray*}
G_7&=&-\frac{\lambda^4}{\bar{v}}e^{-\bar{\phi}}\bar{\phi}_{yy}\sigma_{yy}\big(\frac{1}{\bar{v}+\Phi_y}-\frac{1}{\bar{v}}\big)
+\frac{\lambda^4}{\bar{v}}e^{-\bar{\phi}}\big(\frac{(\bar{\phi}_y+\sigma_y)(\bar{v}_y+\Phi_{yy})}{(\bar{v}+\Phi_y)^2}
-\frac{\bar{\phi}_y\bar{v}_y}{\bar{v}^2}\big)\sigma_{yy}-\frac{\lambda^2}{\bar{v}}\sigma_{yy}F_3\notag\\
&&-\frac{\lambda^4}{\bar{v}}e^{-\bar\phi}\big(\frac{1}{\bar{v}+\Phi_y}-\frac{1}{\bar{v}}\big)\sigma_{yy}^2
-\frac{\lambda^4}{\bar{v}}e^{-\bar\phi}(\frac{1}{\bar{v}}-e^{\bar{\phi}})\sigma_{yy}^2\\
&\sim&O(1)(\bar\phi_{yy}\Phi_y\sigma_{yy}+\bar
v_y\bar\phi_y\Phi_y\sigma_{yy}+\bar\phi_y\Phi_{yy}\sigma_{yy}+\bar
v_y\sigma_{y}\sigma_{yy}+\sigma_y\Phi_{yy}\sigma_{yy}+\Phi_y\sigma_{yy}^2+\bar\phi_{yy}\sigma_{yy}^2\\
&&+\bar
v_y\bar\phi_y\sigma_{yy}^2+\sigma\sigma_{yy}^2+\sigma\Phi_y\sigma_{yy}^2).
\end{eqnarray*}
Then integrating the summation of (\ref{equality421}) and
(\ref{equality422}) over $[0,t]\times\mathbb{R}$ and using Cauchy
inequality and the profile bounds $|\bar{v}_y|+|k_2|\leq C\delta_2$, it yields that
\begin{eqnarray*}
(k_1+k_2)\Phi_y^2+\frac{\lambda^4}{\bar{v}}\sigma^2_{yy}-\frac{2\lambda^2}{\bar{v}}\Phi_y\sigma_{yy}\geq
c(\Phi_y^2+\sigma^2_{yy}),
\end{eqnarray*}
and hence, we arrive at
\begin{eqnarray}
&&\|\Phi_y\|^2+\int_0^t(\|\Phi_y\|^2+\|\sigma_y\|^2+\|\sigma_{yy}\|^2)\mathrm{d}\tau\notag\\
&\leq &C(\|\Phi_{0y}\|^2+\|\Psi_0\|^2+\|\Psi\|^2)
+C\delta_2\int_0^t\|\bar v^{\frac12}\Psi\|^2\mathrm{d}\tau
+C\int_0^t(\|\Psi_y\|^2+\|W_y\|^2)\mathrm{d}\tau\notag\\
&&+C|\int_0^t\int_{\mathbb{R}}(\Phi_y(F_{11}+F_{12})+G_6+G_7)\mathrm{d}y\mathrm{d}\tau|.\label{equality42}
\end{eqnarray}
First, from (\ref{inequality11}), (\ref{inequality12}),
(\ref{assumption31}), the Sobolev inequality and the Young
inequality, one gets
\begin{eqnarray}
|\int_0^t\int_{\mathbb{R}}\Phi_yF_{11}\mathrm{d}y\mathrm{d}\tau|
&\leq&C|\int_0^t\int_{\mathbb{R}}\Phi_y(\Psi^2_y+\Phi_y\Psi_{yy}+\Phi^2_y+\Phi_yw)\mathrm{d}y\mathrm{d}\tau
|\notag\\
&&+C|\int_0^t\int_{\mathbb{R}}\Phi_y(\sigma^2_y+\bar\phi_y\sigma_{y}+\bar\phi_y\Phi_y)\mathrm{d}y\mathrm{d}\tau |\notag\\
&\leq&C\|\Phi_y\|_{L^\infty}\int_0^t\int_{\mathbb{R}}(\Psi^2_y+\Psi_{yy}^2+\Phi^2_y+w^2)\mathrm{d}y\mathrm{d}\tau\notag\\
&&+C\|\sigma_y\|_{L^\infty}\int_0^t\int_{\mathbb{R}}(\Phi^2_y+\sigma^2_y)\mathrm{d}y\mathrm{d}\tau
+C\|\bar\phi_y\|_{L^\infty}\int_0^t\int_{\mathbb{R}}(\Phi^2_y+\sigma^2_y)\mathrm{d}y\mathrm{d}\tau\notag\\
&\le&C(\delta_2+\varepsilon_2)\int_0^t(\|\Phi_y\|^2+\|\sigma_y\|^2+\|\Psi_y\|_1^2+\|w\|^2)\mathrm{d}\tau.\label{estimate421}
\end{eqnarray}
Recalling $F_{12}\sim O(\bar
v_y\bar\phi_y\sigma+\bar\phi_{yy}\sigma+\sigma^2+\bar
v_y\Phi_y+\Phi_y\sigma)$, then similar to (\ref{estimate421}), we
obtain
\begin{eqnarray}
|\int_0^t\int_{\mathbb{R}}\Phi_yF_{12}\mathrm{d}y\mathrm{d}\tau|\le
C(\delta_2+\varepsilon_2)\int_0^t(\|\sigma\|^2+\|\Phi_y\|^2)\mathrm{d}\tau.\label{estimate422}
\end{eqnarray}
Similarly, we can show that
\begin{eqnarray}
|\int_0^t\int_{\mathbb{R}}G_6\mathrm{d}y\mathrm{d}\tau|\le
C(\delta_2+\varepsilon_2)\int_0^t(\|\Phi_y\|_1^2+\|\sigma_y\|_1^2)\mathrm{d}\tau,\label{estimate423}
\end{eqnarray}
and
\begin{eqnarray}
|\int_0^t\int_{\mathbb{R}}G_7\mathrm{d}y\mathrm{d}\tau|\le
C(\delta_2+\varepsilon_2)\int_0^t(\|\Phi_y\|_1^2+\|\sigma_y\|_1^2)\mathrm{d}\tau.\label{estimate424}
\end{eqnarray}

Hence, insertion (\ref{estimate421})-(\ref{estimate424}) into
(\ref{equality42}) and using (\ref{estimate4146}) and Lemma 4.3 lead
to (\ref{estimate42}). This completes the proof.

\

In the following, we proceed to estimate the $H^1$-norm of the
quantities $(\varphi,\psi,w)$. We first rewrite (\ref{model30}) in
the form:
\begin{eqnarray}
\left\{
\begin{array}{lcr}
\varphi_t-s\varphi_y-\psi_y=0, \\
\psi_t-s\psi_y-\frac{1}{\bar{v}}k_3(\bar
v)\varphi_y+\frac{R}{\bar{v}}w_y=\big(\frac{\mu}{\bar{v}}\psi_y\big)_y
+\big(\frac{1}{\bar{v}^2}(R\bar{\theta}-\mu\bar{u}_y)\big)_y\varphi+\frac{1}{\bar{v}}k_2(\bar{v},
\bar{\phi})\varphi_y-(\frac{R}{\bar{v}})_yw \\
~~~~~~~~~~~~~~~~~~~~~~~~~~~~~~~~~~~~~~~~-\frac{\sigma_y}{\bar{v}+\varphi}
-\bar{\phi}_y\big(\frac{1}{\bar{v}+\varphi}-\frac{1}{\bar{v}}\big)+f_1, \\
\frac{R}{\gamma-1}(w_t-sw_y)+k_3(\bar{v})\psi_y=\big(\frac{\kappa}{\bar{v}}w_y\big)_y-\big(\frac{\kappa}{\bar{v}^2}\bar{\theta}_y\varphi\big)_y
+\frac{1}{\bar{v}}\big(\frac{R\bar{\theta}-\mu\bar{u}_y}{\bar{v}}\varphi-Rw+\mu\psi_y\big)\bar{u}_y\\
~~~~~~~~~~~~~~~~~~~~~~~~~~~~~~~~~~~~-k_2(\bar{v},
\bar{\phi})\psi_y+f_2, \\
-\lambda^2\big(\frac{\bar{\phi}_y+\sigma_y}{\bar{v}+\varphi}-\frac{\bar{\phi}_y}{\bar{v}}\big)_y
=-\bar{v}e^{\bar{\phi}}\sigma-\varphi
e^{\bar{\phi}}+f_3,\label{model41}
\end{array}\right.
\end{eqnarray}
here
\begin{eqnarray*}
k_3(\bar
v)&=&\frac{R\theta_\pm}{v_\pm}-s^2(\bar{v}-v_\pm)+\frac{1}{v_\pm\bar{v}}(\bar{v}-v_\pm)>0,\\%~k_4(\bar{v},\bar{\phi})=\frac{\lambda^2}{2}
%(\frac{\bar{\phi}_y}{\bar{v}})^2-\frac{\lambda^2}{\bar{v}}(\frac{\bar{\phi}_y}{\bar{v}})_y,\\
f_1&=&-\big(\frac{\varphi}{\bar{v}(\bar{v}+\varphi)}((k_3(\bar{v})+k_2(\bar{v},\bar{\phi}))\varphi-Rw+\mu\psi_y)\big)_y, \\
f_2&=&\frac{1}{\bar{v}+\varphi}((k_3(\bar{v})+k_2(\bar{v},\bar{\phi}))\varphi-Rw+\mu\psi_y)\big(\psi_y-\frac{1}{\bar{v}}\bar{u}_y\varphi\big)
-\big(\frac{\kappa\varphi}{\bar{v}(\bar{v}+\varphi)}(w_y-\frac{1}{\bar{v}}\bar{\theta}_y\varphi)\big)_y,\\
f_3&=&\bar{v}e^{\bar{\phi}}(1+\sigma-e^\sigma)+\varphi
e^{\bar{\phi}}(1-e^\sigma).
\end{eqnarray*}
\begin{lema}
Under the assumptions of Proposition 3.2, there is a positive
constant $C$ such that it holds for any $t\in[0,T]$,
\begin{eqnarray}
\|(\varphi,\psi,w)(t)\|^2+\int_{0}^{t}\|(\psi_y,w_y)(\tau)\|^2\mathrm{d}\tau&\le&
C\|(\varphi_0,\psi_0,w_0)\|^2+C\int_0^t\|(\psi,\sigma_y)(\tau)\|^2\mathrm{d}\tau\notag\\
&&+C(\delta_2+\varepsilon_2)\int_0^t\|(\varphi,w,\varphi_y)(\tau)\|^2\mathrm{d}\tau.\label{estimate43}
\end{eqnarray}
\end{lema}
\noindent{\bf Proof.} Multiplying (\ref{model41})$_1$,
(\ref{model41})$_2$ and (\ref{model41})$_3$ by $\varphi,
\frac{\bar{v}}{k_3(\bar v)}\psi$ and $\frac{R}{k_3(\bar v)^2}w$,
respectively, the summation of three resultant equalities yields
\begin{eqnarray}
\big(\frac{1}{2}\varphi^2+\frac{\bar{v}}{2k_3(\bar{v})}\psi^2+\frac{R^2}{2(\gamma-1)k_3(\bar{v})^2}w^2\big)_t
-B_{6y}+\frac{\mu}{k_3(\bar{v})}\psi^2_y
+\frac{R\kappa}{\bar{v}k_3(\bar{v})^2}w^2_y=\sum_{i=1}^4g_i
,\label{equality431}
\end{eqnarray}
here
\begin{eqnarray*}
B_6&=&\frac{s}{2}\varphi^2+\frac{s\bar{v}}{2k_3(\bar{v})}\psi^2
+\frac{R^2s}{2(\gamma-1)k_3(\bar{v})^2}w^2+\frac{\mu}{k_3(\bar{v})}\psi\psi_y
-\frac{\varphi}{\bar{v}(\bar{v}+\varphi)}((k_3(\bar{v})+k_2(\bar{v},
\bar{\phi}))\varphi\notag\\
&&-Rw+\mu\psi_y)\frac{\bar{v}}{k_3(\bar{v})}\psi
+\frac{R\kappa}{\bar{v}k_3(\bar{v})^2}ww_y
-\frac{R\kappa}{\bar{v}k_3(\bar{v})^2}\bar{\theta}_y\varphi
w+\varphi\psi-\frac{R}{k_3(\bar{v})}w\psi\notag\\
&&+\frac{R\kappa}{\bar{v}^2k_3(\bar{v})^2}\bar{\theta}_y\varphi
w-\frac{R\kappa}{\bar{v}(\bar{v}+\varphi)k_3(\bar{v})^2}(w_y-\frac{1}{\bar{v}}\bar{\theta}_y\varphi)w,\\
g_1&=&-\frac{\mu}{\bar{v}}(\frac{\bar{v}}{k_3(\bar{v})})_y\psi\psi_y
+(\frac{R}{k_3(\bar{v})})_yw\psi-(\frac{R}{\bar{v}})_y\frac{\bar{v}}{k_3(\bar{v})}w\psi+\frac{k_2(\bar{v},
\bar{\phi})}{k_3(\bar v)}\varphi_y\psi+\frac{\bar{\phi}_y}{k_3(\bar{v})(\bar{v}+\varphi)}\varphi\psi\\
&&-(\frac{R^2s}{2(\gamma-1)k_3(\bar{v})^2})_yw^2-(\frac{s\bar{v}}{2k_3(\bar{v})})_y\psi^2
-\frac{\kappa}{\bar{v}}(\frac{R}{k_3(\bar{v})^2})_yww_y
+\frac{R\kappa}{\bar{v}^2}\bar{\theta}_y\varphi(\frac{1}{k_3(\bar{v})^2})_yw\\
&&-\frac{Rk_2(\bar{v},
\bar{\phi})}{k_3(\bar{v})^2}\psi_yw+\frac{1}{\bar{v}}\big(\frac{R\bar\theta-\mu\bar
u_y}{\bar
v}\varphi-Rw+\mu\psi_y\big)\bar{u}_y\frac{R}{k_3(\bar{v})^2}w+\big(\frac1{\bar
v^2}(R\bar\theta-\mu\bar u_y)\big)_y\frac{\bar
v}{k_3(\bar{v})}\varphi\psi\\
&&+\frac{R\kappa}{\bar{v}(\bar{v}+\varphi)}\big(w_y-\frac{1}{\bar{v}}\bar{\theta}_y\varphi\big)(\frac{1}{k_3(\bar{v})^2})_yw
+\frac{R\kappa}{\bar{v}^2k_3(\bar{v})^2}\bar{\theta}_y\varphi
w_y\\
&&-\frac{R}{\bar{v}k_3(\bar{v})^2(\bar{v}+\varphi)}((k_3(\bar{v})+k_2(\bar{v},
\bar{\phi}))\varphi-Rw+\mu\psi_y)\bar{u}_y\varphi w,\\
g_2&=&-\frac{\bar{v}}{k_3(\bar{v})(\bar{v}+\varphi)}\sigma_y\psi, \notag\\
g_3&=&-\frac{\varphi}{(\bar{v}+\varphi)k_3(\bar{v})}((k_3(\bar{v})+k_2(\bar{v},
\bar{\phi}))\varphi-Rw+\mu\psi_y)\psi_y+\frac{R\kappa}{\bar{v}(\bar{v}+\varphi)k_3(\bar{v})^2}\varphi\big(w_y
-\frac{1}{\bar{v}}\bar{\theta}_y\varphi\big)w_y\\
&&+\frac{R}{k_3(\bar{v})^2(\bar{v}+\varphi)}((k_3(\bar{v})+k_2(\bar{v},
\bar{\phi}))\varphi-Rw+\mu\psi_y)\psi_yw,\notag\\
g_4&=&-\frac{\varphi}{\bar{v}(\bar{v}+\varphi)}((k_3(\bar{v})+k_2(\bar{v},
\bar{\phi}))\varphi-Rw+\mu\psi_y)(\frac{\bar{v}}{k_3(\bar{v})})_y\psi.
\end{eqnarray*}
Then integration (\ref{equality431}) over $[0,t]\times\mathbb{R}$
yields
\begin{eqnarray}
\|(\varphi,\psi,w)\|^2+\int_0^t(\|\psi_y\|^2+\|w_y\|^2)\mathrm{d}\tau
\leq C\|(\varphi_0,\psi_0,w_0)\|^2
+C|\int_0^t\int_{\mathbb{R}}\sum_{i=1}^4g_i\mathrm{d}y\mathrm{d}\tau|.\label{equality43}
\end{eqnarray}
First, since
\begin{eqnarray*}
g_1&\sim&O(1)(\bar v_y\psi\psi_y+\bar
v_yw\psi+\bar\phi_y^2\varphi_y\psi+\bar\phi_{yy}\varphi_y\psi+\bar
v_y\bar\phi_y\varphi_y\psi+\bar\phi_y\varphi\psi+\bar v_yw^2+\bar
v_y\psi^2+\bar v_yww_y\\
&&+\bar v_y\bar\theta_y\varphi
w+\bar\phi_y^2\psi_yw+\bar\phi_{yy}\psi_yw+\bar
v_y\bar\phi_y\psi_yw+\bar u_y\varphi w+\bar u_y\psi_y w+\bar
\theta_y\varphi w_y\\
&&+(\bar\theta_y+\bar u_{yy}+\bar v_y\bar u_y)\varphi\psi),
\end{eqnarray*}
utilizing (\ref{inequality12}) and the Cauchy inequality  to get
\begin{eqnarray}
|\int_0^t\int_{\mathbb{R}}g_1\mathrm{d}y\mathrm{d}\tau|\leq
C\delta_2\int_0^t(\|\varphi\|_1^2+\|\psi\|_1^2+\|w\|_1^2)\mathrm{d}\tau.\label{estimate431}
\end{eqnarray}
Moreover, it is easy to see
\begin{eqnarray}
|\int_0^t\int_{\mathbb{R}}g_2\mathrm{d}y\mathrm{d}\tau|\leq
C\int_0^t(\|\psi\|^2+\|\sigma_y\|^2)\mathrm{d}\tau.\label{estimate432}
\end{eqnarray}
Next, from (\ref{inequality12}), (\ref{assumption31}), the Sobolev
inequality and the Cauchy inequality, one has
\begin{eqnarray}
|\int_0^t\int_{\mathbb{R}}g_3\mathrm{d}y\mathrm{d}\tau|&\leq&C|\int_0^t\int_{\mathbb{R}}(\varphi^2\psi_y+\varphi
w\psi_y+\varphi\psi_y^2+\varphi
w_y^2+\bar\theta_y\varphi^2w_y+w^2\psi_y+w\psi_y^2)\mathrm{d}y\mathrm{d}\tau|\notag\\
&\leq&C(\delta_2+\varepsilon_2)\int_0^t(\|\varphi\|^2+\|w\|_1^2+\|\psi_y\|^2)\mathrm{d}\tau,\label{estimate433}
\end{eqnarray}
and
\begin{eqnarray}
|\int_0^t\int_{\mathbb{R}}g_4\mathrm{d}y\mathrm{d}\tau|&\leq&C|\int_0^t\int_{\mathbb{R}}(\bar
v_y\varphi^2\psi+\bar v_y\varphi w\psi+\bar v_y\varphi \psi\psi_y)\mathrm{d}y\mathrm{d}\tau|\notag\\
&\leq&
C(\delta_2+\varepsilon_2)\int_0^t(\|\varphi\|^2+\|\psi\|_1^2+\|w\|^2)\mathrm{d}\tau.\label{estimate434}
\end{eqnarray}
Therefore, insertion (\ref{estimate431})-(\ref{estimate434}) into
(\ref{equality43}) leads to (\ref{estimate43}). This completes the
proof.

\begin{lema}
Under the assumptions of Proposition 3.2, there is a positive
constant $C$ such that it holds for any $t\in[0,T]$,
\begin{eqnarray}
\|\varphi_y(t)\|^2+\int_{0}^{t}\|\varphi_y(\tau)\|^2\mathrm{d}\tau&\le&
C(\|(\varphi_{0y},\psi_0)\|^2+\|\psi\|^2)+C\int_{0}^{t}
\|(\psi_y,w_y,\sigma_y)(\tau)\|^2\mathrm{d}\tau\notag\\
&&+C(\delta_2+\varepsilon_2)\int_0^t(\|\varphi\|^2+\|\psi\|^2+\|w\|^2+\|\psi_{yy}\|^2)\mathrm{d}\tau.\label{estimate44}
\end{eqnarray}
\end{lema}
\noindent{\bf Proof.} Multiplying (\ref{model41})$_1$ and
(\ref{model41})$_2$ by $\bar{v}\psi_y-\bar{v}_y\psi$ and
$-\bar{v}\varphi_y$, respectively, then the summation of two
resultant equalities leads to
\begin{eqnarray}
\big(\frac{\mu}{2}\varphi^2_y-\bar{v}\psi\varphi_y\big)_t
-\big(\frac{\mu}{2}\varphi_y^2-\bar{v}\varphi_t\psi\big)_y+k_3(\bar
v)\varphi_y^2
=\bar{v}\psi^2_y+\big(Rw_y+\frac{\bar{v}}{\bar{v}+\varphi}\sigma_y-f_1\bar{v}\big)\varphi_y+g_5,\label{equality441}
\end{eqnarray}
here
\begin{eqnarray*}
g_5=-\bar{v}_y\psi(s\varphi_y+\psi_y)-\frac{k_2(\bar
v,\bar\phi_y)}{k_3(\bar v)}\varphi_y^2
-\big(\frac{\mu}{\bar{v}}\big)_y\bar{v}\psi_y\varphi_y
-\big(\frac{k_3(\bar
v)}{\bar{v}}\big)_y\bar{v}\varphi\varphi_y+\big(\frac{R}{\bar{v}}\big)_y\bar{v}w\varphi_y
-\frac{\bar{\phi}_y}{\bar{v}+\varphi}\varphi\varphi_y.
\end{eqnarray*}
Further, integrating  (\ref{equality441}) over
$[0,t]\times\mathbb{R}$ and using the Cauchy inequality, one has
\begin{eqnarray}
\|\varphi_y(t)\|^2+\int_{0}^{t}\|\varphi_y(\tau)\|^2\mathrm{d}\tau&\le&
C(\|(\psi_{0y},\psi_0)\|^2+\|\psi\|^2)+C\int_{0}^{t}
\|(\psi_y,w_y,\sigma_y)(\tau)\|^2)\mathrm{d}\tau\notag\\
&&+C|\int_0^t\int_{\mathbb{R}}(g_5+f_1\varphi_y)\mathrm{d}y\mathrm{d}\tau|.\label{equality44}
\end{eqnarray}
Due to
\begin{eqnarray*}
f_1&\sim&O(1)(\varphi\varphi_y+\varphi\varphi_y(\bar \phi_{yy}+\bar
\phi_y^2+\bar v_y\bar\phi_y)+\varphi_y(w+\psi_y)+\bar
v_y(w+\psi_y)+\varphi^2\varphi_y+\varphi( w_y+\psi_{yy})\\
&&+(\bar\phi_{yy}+\bar \phi_y^2+\bar
v_y\bar\phi_y)\varphi^2\varphi_y+(\bar \phi_y\bar\phi_{yy}+\bar
v_y\bar\phi_{y}+\bar\phi_{yyy}+\bar v_y\bar\phi_{yy}+\bar \phi_y\bar
v_{yy}+\bar v_{y})\varphi^2) ,
\end{eqnarray*}
and
\begin{eqnarray*}
g_5&\sim&O(1)(\bar v_y\psi\varphi_y+\bar
v_y\psi\psi_y+(\bar\phi_y^2+\bar\phi_{yy}+\bar
v_y\bar\phi_y)\varphi_y^2+\bar v_y\varphi_y\psi_y+\bar
v_y\varphi\varphi_y+\bar v_yw\varphi_y+\bar\phi_y\varphi\varphi_y),
\end{eqnarray*}
then same as (\ref{estimate431}) and (\ref{estimate433}), we can
obtain
\begin{eqnarray}
|\int_0^t\int_{\mathbb{R}}g_5\mathrm{d}y\mathrm{d}\tau|\leq
C(\delta_2+\varepsilon_2)\int_0^t(\|\varphi\|_1^2+\|\psi\|_1^2+\|w\|^2)\mathrm{d}\tau,\label{estimate441}
\end{eqnarray}
and
\begin{eqnarray}
|\int_0^t\int_{\mathbb{R}}f_1\varphi_y\mathrm{d}y\mathrm{d}\tau|\leq
C(\delta_2+\varepsilon_2)\int_0^t(\|\varphi_y\|^2+\|\psi_y\|_1^2+\|w_y\|^2)\mathrm{d}\tau.\label{estimate442}
\end{eqnarray}
Therefore, from (\ref{equality44}), (\ref{estimate441}) and
(\ref{estimate442}), we immediately show (\ref{estimate44}). This
completes the proof.

\begin{lema}
Under the assumptions of Proposition 3.2, there is a positive
constant $C$ such that it holds for any $t\in[0,T]$,
\begin{eqnarray}
\|(\psi_y,w_y)(t)\|^2+\int_{0}^{t}\|(\psi_{yy},w_{yy})(\tau)\|^2\mathrm{d}\tau&\le&
C\|(\psi_{0y},w_{0y})\|^2+C\int_{0}^{t}
\|(\varphi_y,\psi_y,w_y)(\tau)\|^2\mathrm{d}\tau\notag\\
&&+C(\delta_2+\varepsilon_2)\int_0^t\|(\varphi,w,\varphi_y)(\tau)\|^2d\tau.\label{estimate45}
\end{eqnarray}
\end{lema}
\noindent{\bf Proof.} Multiplying (\ref{model41})$_2$ and
(\ref{model41})$_3$ by $-\psi_{yy}$ and $-w_{yy}$, respectively, the
summation of two resultant equalities yields
\begin{eqnarray}
&&\big(\frac{1}{2}\psi^2_y+\frac{R}{2(\gamma-1)}w^2_y\big)_t-\big(\psi_t\psi_y-\frac{s}{2}\psi^2_y
-\frac{Rs}{2(\gamma-1)}w^2_y+\frac{R}{\gamma-1}w_tw_y\big)_y+\frac{\mu}{\bar{v}}\psi_{yy} +\frac{\kappa}{\bar{v}}w^2_{yy}\notag\\
&=&-\frac{k_3(\bar
v)}{\bar{v}}\varphi_y\psi_{yy}-\frac{R}{\bar{v}}w_y\psi_{yy}+k_3(\bar
v)\psi_yw_{yy}+\frac{\sigma_y}{\bar{v}+\varphi}\psi_{yy}
+g_6-f_1\psi_{yy}-f_2w_{yy},\label{equality451}
\end{eqnarray}
and
\begin{eqnarray*}
g_6&=&-(\frac{\mu}{\bar{v}})_y\psi_y\psi_{yy}-(\frac{k_3(\bar
v)}{\bar{v}})_y\varphi\psi_{yy}
+(\frac{R}{\bar{v}})_yw\psi_{yy}+\bar{\phi}_y\big(\frac{1}{\bar{v}+\varphi}-\frac{1}{\bar{v}}\big)\psi_{yy}
-(\frac{\kappa}{\bar{v}})_yw_yw_{yy}\\
&&+(\frac{\kappa}{\bar{v}^2}\bar{\theta}_y\varphi)_yw_{yy}
-\frac{1}{\bar{v}}\big(\frac{R\bar\theta-\mu\bar u_y}{\bar
v}\varphi-Rw+\mu\psi_y\big)\bar{u}_yw_{yy}-\frac{k_2(\bar
v,\bar\phi_y)}{\bar{v}}\varphi_y\psi_{yy}+k_2(\bar
v,\bar\phi_y)\psi_yw_{yy}.
\end{eqnarray*}
Then integrating  (\ref{equality451}) over $[0,t]\times\mathbb{R}$
and using the Cauchy inequality, we have
\begin{eqnarray}
\|(\psi_y,w_y)(t)\|^2+\int_{0}^{t}\|(\psi_{yy},w_{yy})(\tau)\|^2\mathrm{d}\tau&\le&
C\|(\psi_{0y},w_{0y})\|^2+C\int_{0}^{t}
\|(\varphi_y,\psi_y,w_y)(\tau)\|^2\mathrm{d}\tau\notag\\
&&+C|\int_0^t\int_{\mathbb{R}}(g_6+f_1\psi_{yy}+f_2w_{yy})\mathrm{d}y\mathrm{d}\tau|.\label{equality45}
\end{eqnarray}
Moreover, we also note that
\begin{eqnarray*}
f_2&\sim&O(1)(\varphi\psi_y(\bar \phi_{yy}+\bar \phi_y^2+\bar
v_y\bar\phi_y)+\psi_y(\varphi+w+\psi_y)+\varphi^2(\bar u_y+\bar
\phi_y^2\bar u_y+\bar u_y\bar\phi_{yy}+\bar v_y\bar
u_y\bar\phi_{y}+\bar\theta_{yy}\\
&&+\bar v_y\bar\theta_y)+\bar u_y\varphi (w+\psi_y)+(\varphi_y+\bar
v_y)w_y+(\bar\theta_y+w_y)\varphi\varphi_y+\bar
v_y\bar\theta_y\varphi+\bar\theta_y\varphi^2\varphi_y+\varphi
w_{yy}) ,
\end{eqnarray*}
and
\begin{eqnarray*}
g_6&\sim&O(1)(\bar v_y\psi_y\psi_{yy}+(\bar
v_y+\bar\phi_y)\varphi\psi_{yy}+\bar v_yw\psi_{yy}+\bar
v_yw_yw_{yy}+\bar\theta_y\varphi_yw_{yy}+(\bar\theta_{yy}+\bar
v_y\bar\theta_y+\bar u_y\\
&&+\bar u_y^2)\varphi w_{yy}+\bar
u_y(w+\psi_y)w_{yy}+(\bar\phi_y^2+\bar\phi_{yy}+\bar
v_y\bar\phi_y)\varphi_y\psi_{yy}+(\bar\phi_y^2+\bar\phi_{yy}+\bar
v_y\bar\phi_y)\psi_yw_{yy}).
\end{eqnarray*}
Furthermore, similar to (\ref{estimate431}), (\ref{estimate433}) and
(\ref{estimate434}), one gets
\begin{eqnarray}
|\int_0^t\int_{\mathbb{R}}g_6\mathrm{d}y\mathrm{d}\tau|\leq
C\delta_2\int_0^t(\|\varphi\|_1^2+\|w\|_2^2+\|\varphi_y\|^2+\|\psi_y\|_1^2)\mathrm{d}\tau,\label{estimate451}
\end{eqnarray}
\begin{eqnarray}
|\int_0^t\int_{\mathbb{R}}f_1\psi_{yy}\mathrm{d}y\mathrm{d}\tau|\leq
C\varepsilon_2\int_0^t(\|\varphi_y\|^2+\|\psi_y\|_1^2+\|w_y\|^2)\mathrm{d}\tau,\label{estimate452}
\end{eqnarray}
and
\begin{eqnarray}
|\int_0^t\int_{\mathbb{R}}f_2w_{yy}\mathrm{d}y\mathrm{d}\tau|\leq
C(\delta_2+\varepsilon_2)\int_0^t(\|\varphi\|_1^2+\|\psi_y\|_1^2+\|w_y\|_1^2)\mathrm{d}\tau.\label{estimate453}
\end{eqnarray}
Hence, putting (\ref{estimate451})-(\ref{estimate453}) into
(\ref{equality45}), we immediately have (\ref{estimate45}). This
completes the proof.

\

\noindent {\bf Proof of Proposition~3.2.} Summing up the estimates
(\ref{estimate41}), (\ref{estimate42}), (\ref{estimate43}),
(\ref{estimate44}) and (\ref{estimate45}), and using Lemma 4.1 and
(\ref{estimate4146}), we immediately have (\ref{inequality33}). This
completes the proof.

\

\noindent {\bf Acknowledgements:}  The research is supported in part
by the National Natural Science Foundation of China (Grant Nos. 12471210, 12331007 and 12671260).

\end{document}